\documentclass[10pt]{article}

\usepackage{amsmath,amsfonts,amssymb,nccmath}

\usepackage{cancel}
\usepackage{tikz-cd}
\usetikzlibrary{arrows}

\usepackage{float}

\usepackage[numbers]{natbib}
\usepackage{graphicx}

\usepackage{hyperref}
\hypersetup{
   colorlinks   =  false,
    linkcolor    = cyan,
    citecolor    = red,
     urlcolor	=magenta,     
}

\usepackage{amsthm}
\theoremstyle{definition} 
\newtheorem{example}{Example}
\newtheorem{definition}{Definition}
\theoremstyle{theorem} 

\newtheorem{proposition}{Proposition}
\newtheorem{corollary}{Corollary}

\newcommand{\red}[1]{{\color{red}#1}}

\def\be{\begin{equation}}
\def\ee{\end{equation}}
\def\bea{\begin{eqnarray}}
\def\eea{\end{eqnarray}}

\begin{document}

\begin{center}
\baselineskip 24 pt {\LARGE \bf  
Hamiltonian structure of compartmental epidemiological models}
\end{center}

\begin{center}

{\sc Angel Ballesteros$^1$, Alfonso Blasco$^1$, Ivan Gutierrez-Sagredo$^{1,2}$}

\medskip

{$^1$Departamento de F\'isica, Universidad de Burgos, 
09001 Burgos, Spain}

{$^2$Departamento de Matem\'aticas y Computaci\'on, Universidad de Burgos, 
09001 Burgos, Spain}

 \medskip
 
e-mail: {\href{mailto:angelb@ubu.es}{angelb@ubu.es}, \href{mailto:ablasco@ubu.es}{ablasco@ubu.es}, \href{mailto:igsagredo@ubu.es}{igsagredo@ubu.es}}

\end{center}

\medskip

\begin{abstract}

Any epidemiological compartmental model with constant population is shown to be a Hamiltonian dynamical system in which the total population plays the role of the Hamiltonian function. Moreover, some particular cases within this large class of models are shown to be bi-Hamiltonian. New interacting compartmental models among different populations, which are endowed with a Hamiltonian structure, are introduced. The Poisson structures underlying the Hamiltonian description of all these dynamical systems are explicitly presented, and their associated Casimir functions are shown to provide an efficient tool in order to find exact analytical solutions for epidemiological models, such as the ones describing the dynamics of the COVID-19 pandemic.

\end{abstract}

\bigskip


\tableofcontents

\section{Introduction}

The recent COVID--19 pandemic has showed that a deeper understanding of the dynamics of the outbreaking and spreading of an infectious disease is necessary \cite{MB2020science,Fanelli2020,Barmparis2020,NANT2020wuhan,YW2020wuhan,KGL2020covid}. Specially urgent is the need for models that accurately predicts the first stages of an epidemic outbreak \cite{Chowell2016}, in order to guide a safe reopening of countries and thus avoiding new outbreaks that could be devastating \cite{Prem2020}. 

The origin of the mathematical modeling for epidemics can be traced back to the works of Daniel Bernouilli regarding the spreading of smallpox \cite{Bernoulli1760}. Since then the field of mathematical epidemiology has experienced a great development, mainly in the form of very concrete models describing epidemics with specific features \cite{Korobeinikov2005,Miller2012}. Among them, the so-called compartmental models \cite{Hethcote1973,Li1999,Hethcote2007,Boujakjian2016,Miller2017} are based on the classification of the individuals of a given total population in a number of compartments, each of them representing a different stage in the infectious process. In these models, population transitions among different compartments can be assumed to be either deterministic or stochastic. For deterministic models, the movement from one compartment to another is governed by an ordinary differential equation. So, mathematically, these compartmental models are just systems of (nonlinearly) coupled ordinary differential equations (ODEs). 

Among these models, a prominent role is played by the original SIR model proposed by Kermack and McKendrick \cite{KM1927sir}, which is expressed by a 3-dimensional dynamical system defined by
\begin{equation}
\label{eq:originalSIR}
\begin{split}
&\dot S(t) = - \beta\, S(t) \, I(t) , \\
&\dot I(t) = \beta \,S(t)\, I(t) - \alpha \, I(t), \\
&\dot R(t) = \alpha \, I(t) , \\
\end{split}
\end{equation}
which is subjected to the initial conditions $S(0)=S_0$, $I(0)=I_0$ and $R(0)=R_0$. In this very schematic model the population is divided into 3 compartments: 
\begin{itemize}

\item Susceptible individuals $S(t)$: not infected individuals at time $t$ who could potentially be infected.

\item Infected individuals $I(t)$: infected individuals at time $t$ which can spread the disease at time $t' > t$. 

\item Recovered individuals $R(t)$: this compartment really contains 2 subgroups:

\begin{itemize}

\item Individuals that have been infected at time $t' < t$ in the past and have developed immunization to the disease.

\item Individuals that have died at time $t' < t$.

\end{itemize}

\end{itemize}

The original SIR model assumes that both the transmission rate $\beta$ (the probability of any susceptible individual being infected when meeting an infected individual) and the recovery rate $\alpha$ are constants. Moreover, it is also assumed that the immunization is perpetual, so recovered individuals cannot be reinfected. Of course, since there is no transference from the recovered compartment to the other two ones, deaths are considered as 'recovered' individuals. Moreover, the SIR model also assumes that the total population is constant, since
\begin{equation}
\dot S + \dot I + \dot R = 0 ,
\end{equation}
and thus 
\begin{equation}
S+I+R = N .
\end{equation}
Note that there is no loss of generality in assuming that $N=1$ and therefore thinking of $S$, $I$ and $R$ as fractions of the total population. This will be the case during the rest of the paper, unless otherwise stated.

To be more precise, this model should be interpreted as describing an epidemic whose dynamics is fast enough in order to assume that the size of the population is constant (so no `vital dynamics' are considered) and such that the recovered individuals are immunized for a long enough time. In this way the change in the total population by causes different from the infection can be neglected. The dynamics of this simple model, which is given in Figure \ref{fig:originalSIR}, can be schematically represented by the diagram depicted in Figure \ref{diag:originalSIR}. Note that in this kind of diagrams, an arrow going out from a given compartment implies a negative sign for the corresponding term in the r.h.s. of the differential equation for the derivative of the population for such compartment. Since each arrow provides two terms with opposite signs, any model constructed through an arbitrary number of arrows between pairs of compartments will be such that the total population is constant (the sum over all compartments of all the terms at the r.h.s. of the system of ODEs vanishes). 

\begin{figure}[H]
\begin{minipage}{.5\textwidth}
\begin{center}
\begin{tikzpicture}[->]
\node (s) at (-2,0) {$S$};
\node (i) at (2,0) {$I$};
\node (r) at (0,-2) {$R$};
\draw[blue] (s) to [bend left=20] node[midway,above] {$\beta S I$} (i);
\draw[blue] (i) to [bend left] node[midway,below] {$\alpha I$} (r);
\end{tikzpicture}
\caption{Schematic diagram for the SIR system. \label{diag:originalSIR}}
\end{center}
\end{minipage}
\begin{minipage}{.5\textwidth}
\includegraphics[width=7.7cm, height=4.6cm]{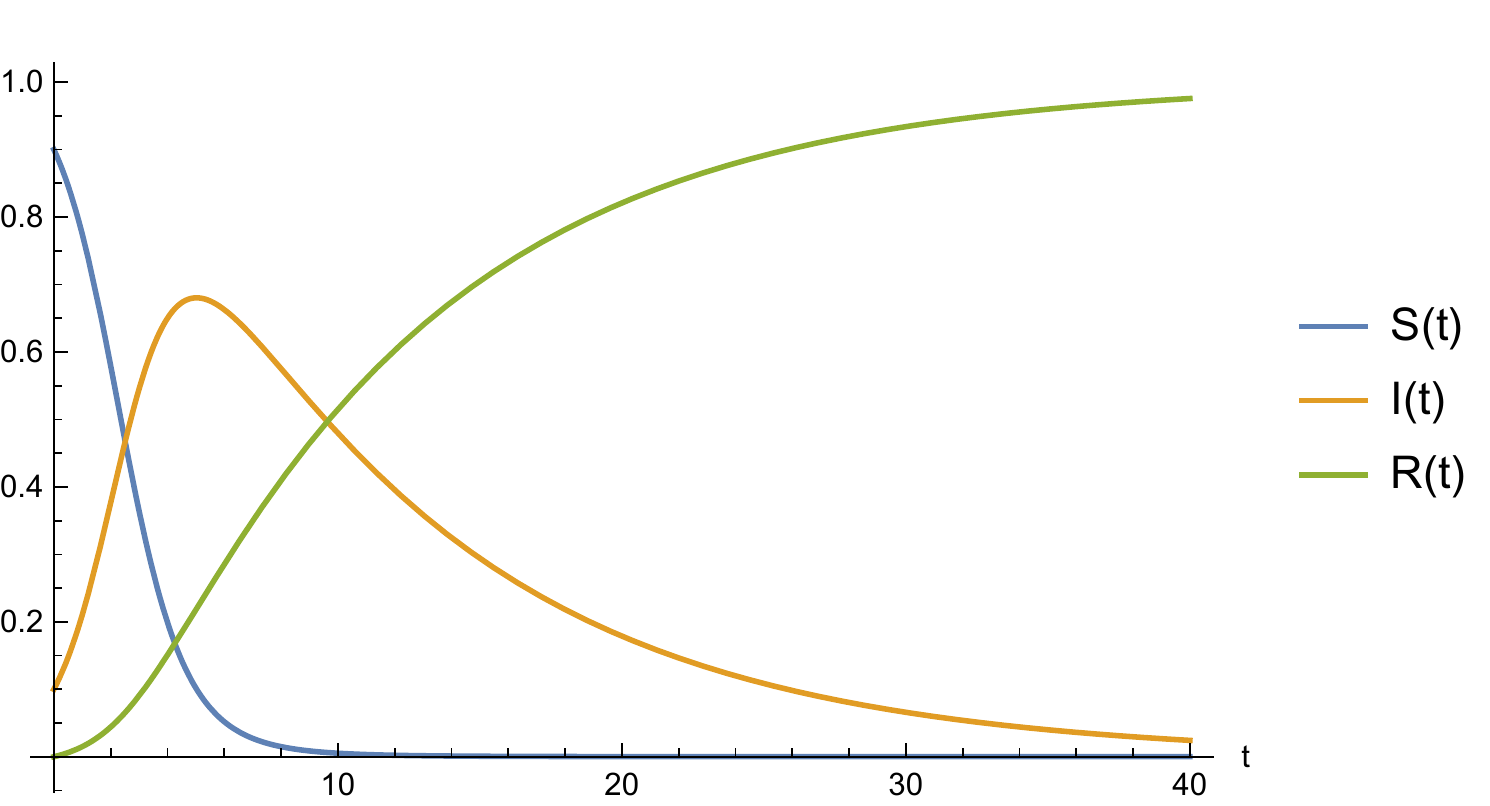} 
\caption{Typical dynamics of the SIR system for $\beta=1, \alpha =0.1$. The infection grows to reach its peak, and then the infection disappears. \label{fig:originalSIR}}
\end{minipage}
\end{figure}

Despite its great simplicity, the SIR model of Kermack and McKendric has proved to provide the essential qualitative features of several complex epidemics. In fact, it has been proved to give an accurate quantitative  description of the COVID--19 pandemic for some countries \cite{Postnikov2020sir}. A generalization of the original SIR model that is sensible for viral infections \cite{CG1978generalizationSIR,Brauer1990} (see also \cite{Brauer2005,BrauerCastillo2012book}) consists in substituting the constant transmission rate by a function that depends on the sum of susceptible and infected populations, namely
\begin{equation}
\label{eq:betaSI}
\beta(S+I)=\frac{\phi(S+I)}{S+I} ,
\end{equation}
with $\phi : \mathbb R_+ \to \mathbb R_+$ an strictly positive function such that
\begin{equation}
\phi'(x) \geq 0 ,\qquad \left( \frac{\phi(x)}{x} \right)' \leq 0 .
\end{equation}
In what follows we will consider this generalization, since all our results will also hold in this case.

A large number of compartmental models have been proposed, most of them being extensions of the original SIR system since more complex epidemics may require extra terms in the system of ODEs. For instance, endemic diseases characterized by a non-permanent immunization could be considered, in which an extra arrow connecting the $R$ and $S$ compartments has to appear in the diagram generalizing Figure \ref{diag:originalSIR}. Other diseases may present a passive immunity, like maternal passive immunity, in which newborns receive maternal antibodies. This feature is typically modeled by adding an extra compartment to the model. All these models have been shown to be successful in order to model complex epidemics, and they are usually solved  through standard numerical techniques for nonlinear systems of ODEs. 

It is also well-known \cite{Nutku1990sir} that the original SIR system \eqref{eq:originalSIR} admits a Hamiltonian description (for the sake of completeness, a summary of the basics on the theory of Hamiltonian dynamical systems is provided in the Appendix). In this approach, the Hamiltonian function is just the total population
\begin{equation}
\mathcal H_{1}=S+I+R \, ,
\end{equation}
and the the SIR dynamical system is obtained as the Hamilton's equations
\begin{equation}
\label{eq:Hamilton_eqs}
\dot S = \{S,\mathcal H_{1} \}_1 , \qquad
\dot I = \{I,\mathcal H_{1} \}_1 , \qquad
\dot R = \{R,\mathcal H_{1} \}_1 ,
\end{equation}
 with respect to the Poisson structure \begin{equation}
\label{eq:Poisson_original_SIR}
\{S,I\}_{1}=0, \quad \{S,R\}_{1}=-\beta S\,I, \quad \{I,R\}_{1}=-\alpha\,I+\beta\,S\,I .
\end{equation}
Moreover, Nutku also proved \cite{Nutku1990sir} that the SIR system is bi-Hamiltonian, since there does exist a second Poisson structure 
\begin{equation}
\{S,I\}_{2}=-\beta\,S\,I, \quad \{S,R\}_{2}=\beta S\,I, \quad \{I,R\}_{2}=-\beta\,S\,I ,
\end{equation}
which is compatible with the first one (in the sense that $\{\cdot,\cdot \} = \{\cdot,\cdot \}_1 + \{\cdot,\cdot \}_2$ is a well-defined Poisson structure, see the Appendix), together with a second Hamiltonian function which is not the total population, namely
\begin{equation}
\label{eq:H_2_originalSIR}
\mathcal H_{2}=-R-\dfrac{\alpha}{\beta}\log(S) \, ,
\end{equation}
but nevertheless the SIR dynamical system can be again written as the Hamilton's equations 
\begin{equation}
\dot S = \{S,\mathcal H_{2} \}_2 = -\beta S I, \qquad
\dot I =  \{I,\mathcal H_{2} \}_2 = \beta S I - \alpha I, \qquad
\dot R = \{R,\mathcal H_{2} \}_2 = \alpha I. 
\end{equation}
From the previous expressions it is easy to check that $\{ f,\mathcal H_{1} \}_2 =0$ for all $f(S,I,R)$, so $\mathcal H_{1}$ is a Casimir function for $\{\cdot,\cdot \}_2$. A Casimir function for $\{ \cdot,\cdot \}_1$ is given by 
\begin{equation}
\mathcal C_1 = S+I-1-\dfrac{\alpha}{\beta}\log(S) \, ,
\end{equation}
which coincides with \eqref{eq:H_2_originalSIR} in the hypersurface of constant population $\mathcal M = \mathcal H_1^{-1} (1)$, i.e. $\mathcal H_2 |_{\mathcal M} = \mathcal C_1 |_{\mathcal M}$. In this case we could have taken $\mathcal C_1$ as the second Hamiltonian function. However, we choose not to do so since this will no longer be true in some of the examples presented in this paper. The fact that Casimir functions provide additional integrals for the dynamics is one of the essential features of the Hamiltonian description of dynamical systems.

The aim of this paper is to show that a Hamiltonian structure in which the total population plays the role of the Hamiltonian function can be defined for a very large class of (deterministic) compartmental epidemiological models, including certain systems of interacting  populations which, to the best of our knowledge, have not been previously considered in the literature. In general, interacting models should be relevant in the context of the current COVID--19 pandemic, where some regions/countries can be described by the same dynamical models of the infection, although with different parameters, and the exchange of individuals among regions must be taken into account for the joint evolution of the pandemic. We stress that, in general, the Poisson structures defining the Hamiltonian structures here presented could be helpful in order to find exact analytical solutions for the corresponding epidemiological models provided that the associated Casimir functions are found. 

The structure of the paper is the following. In Section \ref{sec:generalizedSIR} we study the most general compartmental model with three compartments and constant population, which we call the generalized SIR model, and we show that this is indeed a Hamiltonian system. Some outstanding instances of this model  are considered and, moreover, we show that some compartmental models with non-constant population are equivalent to this system. In Section \ref{sec:Hamiltoniancompmodels} we prove our main result stating that every compartmental model with constant population is Hamiltonian. In Section \ref{sec:interactingSIR} we introduce a family of models consisting in an arbitrary number of generalized SIR systems whose compartments also interact by the exchange of individuals from one population to another, and the Hamiltonian structure for these interacting models with constant joint total population  is presented. In Section \ref{sec:interacting_compartmental} we show that an analogous result is valid, under a mild assumption, by taking arbitrary compartmental models as building blocks for the interacting system. In Section \ref{sec:biHamiltonian} we explicitly present the bi-Hamiltonian structure of several compartmental models, in particular of the SIR model with a vaccination function and the endemic SIRS system. Finally, in Section \ref{sec:exactsol} we find exact analytical solutions for some of the examples previously considered, which reduce to the exact analytical solution of the SIR model when the relevant parameters vanish. A Section including some comments and open questions closes the paper.

\section{Generalized SIR model}
\label{sec:generalizedSIR}

The basic SIR system \eqref{eq:originalSIR} of Kermack and McKendrick can be extended in order to include more general interactions among the susceptible, infected and recovered compartments. In fact, the most general 3-dimensional dynamical system containing the original SIR (with the generalized transmission rate $\beta(S+I)$ given by \eqref{eq:betaSI}) and maintaining the total population constant, which we call the \emph{generalized SIR system}, can be written as
\begin{equation}
\label{eq:generalizedSIR}
\begin{split}
&\dot S = - \beta S I - \varphi_{S \to I} (S,I,R) - \varphi_{S \to R} (S,I,R), \\
&\dot I = \beta S I - \alpha I + \varphi_{S \to I} (S,I,R) - \varphi_{I \to R} (S,I,R), \\
&\dot R = \alpha I + \varphi_{S \to R} (S,I,R) + \varphi_{I \to R} (S,I,R), \\
\end{split}
\end{equation}
where $\varphi_{A \to B} (S,I,R)$ represents the transfer function from the compartment $A$ to the compartment $B$, which in general are allowed to be arbitrary functions of $S$, $I$ and $R$. Obviously the original terms of the SIR system could be absorbed in $\varphi_{A \to B} (S,I,R)$, but for the sake of clarity we choose not to do so. This system is represented by the following schematic diagram:
\begin{center}
\begin{tikzpicture}[->]
\node (s) at (-3,0) {$S$};
\node (i) at (3,0) {$I$};
\node (r) at (0,-3) {$R$};
\draw[blue] (s) to [bend left=20] node[midway,above] {$\beta S I$} (i);
\draw (s) to [bend right=20] node[midway,above] {$\varphi_{S \to I} (S,I,R)$} (i);
\draw[blue] (i) to [bend left] node[midway,below] {$\alpha I$} (r);
\draw (s) to [bend right] node[midway,left] {$\varphi_{S \to R} (S,I,R)$} (r);
\draw (i) to [bend right=20] node[midway,below] {$\varphi_{I \to R} (S,I,R)$} (r);
\end{tikzpicture}
\end{center}
where the direction of an arrow simply shows the transference of individuals given by the positive part of the corresponding function. Of course, composition of arrows is allowed, and blue arrows are the corresponding ones to the original SIR model.

Indeed, such a generic system does not need to be biologically relevant, since not all choices for $\varphi_{A \to B} (S,I,R)$ describe sensible epidemiological models. However, different choices of these functions do lead to systems catching key features of different epidemiological situations. In any case, in the sequel we show that such a generalized SIR \eqref{eq:generalizedSIR} system admits a unified description as a Hamiltonian system.

\subsection{A unified Hamiltonian description for the generalized SIR model}

For concreteness, we consider the generalized SIR model \eqref{eq:generalizedSIR}, with $\beta=\beta(S+I)$ of the form \eqref{eq:betaSI} and $\alpha$ constant. However we should stress that the following result is valid for $\alpha(S,I)$ and $\beta(S,I)$ arbitrary functions of $S$ and $I$. We have the following 

\begin{proposition}
\label{prop:generalizedSIRhamiltonian}
The generalized SIR system given by \eqref{eq:generalizedSIR} is a Hamiltonian system, with respect to the Hamiltonian function 
\begin{equation}
\label{eq:H_SIR}
\mathcal H(S,I,R) = S + I + R
\end{equation}
and the Poisson algebra defined by the fundamental Poisson brackets given by
\begin{equation}
\begin{split}
\label{eq:poissonSI0}
&\{ I,R\} = -\alpha I + \beta S I + \varphi_2(S,I) , \\
&\{ S,R\} = - \beta S I + \varphi_1(S,I) , \\
&\{ S,I\} = 0, \\
\end{split}
\end{equation}
with
\begin{equation}
\label{eq:varphi12_SIR}
\begin{split}
& \varphi_1 = -\varphi_{S \to I} - \varphi_{S \to R} , \\
& \varphi_2 = \varphi_{S \to I} - \varphi_{I \to R} . \\
\end{split}
\end{equation}
\end{proposition}

\begin{proof}

Hamilton's equations for \eqref{eq:H_SIR} and \eqref{eq:poissonSI0} read
\begin{equation}
\begin{split}
\label{eq:generalSIR}
&\dot S = \{S,\mathcal H \} = - \beta S I + \varphi_1(S,I) , \\
&\dot I = \{I,\mathcal H \} = \beta S I - \alpha I + \varphi_2(S,I), \\
&\dot R = \{R,\mathcal H \} = \alpha I - \varphi_1(S,I) - \varphi_2(S,I), \\
\end{split}
\end{equation}
which agrees with \eqref{eq:generalizedSIR} under the identification \eqref{eq:varphi12_SIR}. Moreover, it is easy to check that $\{\cdot,\cdot \}$ is a well-defined Poisson structure, since the Jacobi identity
\begin{equation}
\mathrm{Jac}(S,I,R) = \{\{S,I\},R\} + \{\{R,S\},I\} + \{\{I,R\},S\} = 0 \, ,
\end{equation}
holds. The total phase space for the system  is thus $\mathbb R^3_+$. However, the fact that the Hamiltonian function $\mathcal H : \mathbb R_+^3 \to \mathbb R_+$ is conserved along trajectories
\begin{equation}
\dot{ \mathcal H}=\dot{S}+\dot{I}+\dot{R} = 0 ,
\end{equation}
together with the fact that the Hamiltonian is just the total population implies that the dynamics of the system takes place on $\mathcal M = \mathcal H^{-1}(1)$ (recall that local coordinates $S$, $I$, $R$ are interpreted as fractions of the total population, which is conserved). Therefore, the restriction of any function $F : \mathbb R_+^3 \to \mathbb R_+$ to the hypersurface $\mathcal M$ is given by $F|_{\mathcal M} = F(S,I, 1-I-S)$, thus proving that the Hamiltonian system allows to include arbitrary functions of $R$.
\end{proof}

With this notation, which is the one that we will use throughout the paper, the explicit dependence of $R$ can be avoided, and the generalized SIR model is represented by the following diagram:
\begin{center}
\begin{tikzpicture}[->]
\node (s) at (-2,0) {$S$};
\node (i) at (2,0) {$I$};
\node (r) at (0,-2) {$R$};
\draw[blue] (s) to [bend left=20] node[midway,above] {$\beta S I$} (i);
\draw[blue] (i) to [bend left] node[midway,below] {$\alpha I$} (r);
\draw (r) to [bend left] node[midway,below] {$\varphi_1(S,I)$} (s);
\draw (r) to [bend left] node[midway,above] {$\varphi_2(S,I)$} (i);
\end{tikzpicture}
\end{center}
Note that the fact that $\varphi_1(S,I)$ and $\varphi_2(S,I)$ can be arbitrary functions of $S,I$, while satisfying $\mathrm{Jac}(S,I,R) =0$, can be traced back to the fact that $\{S,I\}$ vanishes \eqref{eq:poissonSI0}.

The following particular examples will play an important role in the rest of the paper.

\begin{example}[Original SIR model]
\label{ex:OriginalSIR}
The original SIR system \eqref{eq:originalSIR} is recovered by setting $\beta(S+I) = \beta$ and $\varphi_1(S,I) = \varphi_2(S,I) = 0$.
\hfill$\diamondsuit$
\end{example}

\begin{example}[Endemic SIRS model]
\label{ex:SIRSendemic}

An endemic disease is characterized by a non-permanent immunization, in such a way that the individuals that have been infected can be re-infected after a sufficiently large amount of time. A simple model capturing this feature is recovered from \eqref{eq:generalSIR} by setting  $\varphi_1(S,I) = \mu I$ and $\varphi_2(S,I) = -\mu I$. The ODE system becomes
\begin{equation}
\label{eq:SIRSendemic}
\begin{split}
&\dot S = - \beta S I + \mu\,I , \\
&\dot I = \beta S I - (\alpha+\mu) I, \\
&\dot R = \alpha I . \\
\end{split}
\end{equation}

The SIRS endemic model, whose typical dynamics is shown in Figure \ref{fig:endemicSIRS}, is represented by the diagram of Figure \ref{diag:endemicSIRS}.
\hfill$\diamondsuit$

\begin{figure}[H]
\begin{minipage}{.5\textwidth}
\begin{center}
\begin{tikzpicture}[->]
\node (s) at (-2,0) {$S$};
\node (i) at (2,0) {$I$};
\node (r) at (0,-2) {$R$};
\draw[blue] (s) to [bend left=20] node[midway,above] {$\beta S I$} (i);
\draw[blue] (i) to [bend left] node[midway,below] {$\alpha I$} (r);
\draw (r) to [bend left] node[midway,below] {$\mu I$} (s);
\draw[<-] (r) to [bend left] node[midway,above] {$\mu I$} (i);
\end{tikzpicture}
\caption{Endemic SIRS diagram. \label{diag:endemicSIRS}}
\end{center}
\end{minipage}
\begin{minipage}{.5\textwidth}
\includegraphics[width=7.7cm, height=4.6cm]{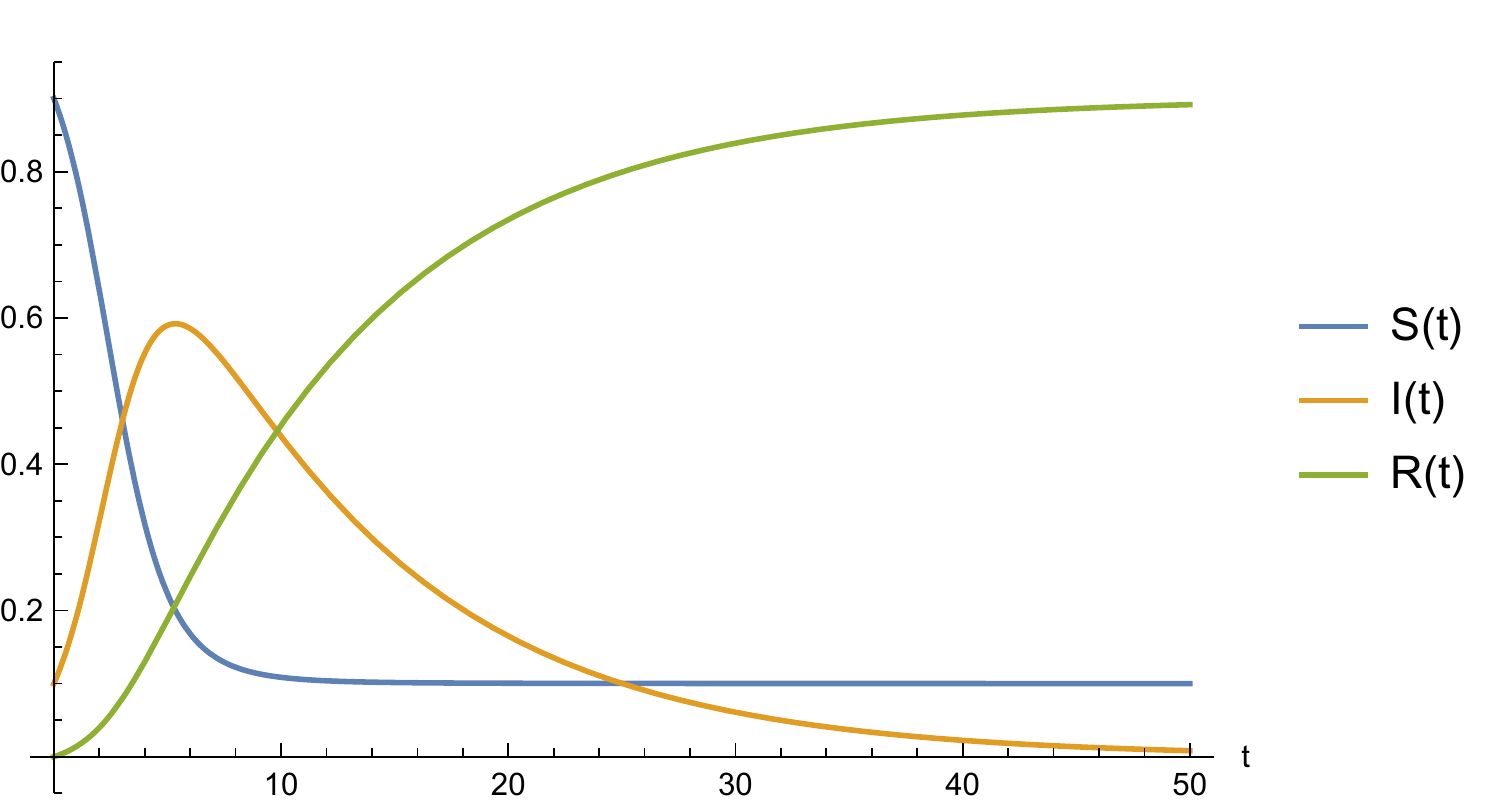} 
\caption{Typical dynamics of the endemic SIRS system ($\beta=1, \alpha =0.1, \mu = 0.1$). Similarly to the original SIR (Figure \ref{fig:originalSIR}) the infection disappears. However, due to the temporal immunization, the susceptible population does not vanish for $t$ large, thus making future outbreaks possible.  \label{fig:endemicSIRS}}
\end{minipage}
\end{figure}

\end{example}

\begin{example}[SIR model including a vaccination function]
\label{ex:SIRvaccination}

Systems that model vaccinated populations are specially interesting, since they allow to design optimal vaccination strategies (see for example \cite{Banerjee2020,Buonomo2020}). Setting $\varphi_1(S,I)=- v(S,I)$ and $\varphi_2(S,I)=0$ in the generalized SIR \eqref{eq:generalSIR} we have the system 
\begin{equation}
\begin{split}
\label{eq:SIRvaccination_general}
&\dot S = - \beta S I - v(S,I) , \\
&\dot I = \beta S I - \alpha I , \\
&\dot R = \alpha I + v(S,I) . \\
\end{split}
\end{equation}
If $v : \mathbb R_+^2 \to \mathbb R_+$ is a strictly positive function, this model describes a population that is receiving a vaccination during the epidemic outbreak. The diagram of a SIR model including vaccination, together with an example of its dynamics are given in Figures \ref{diag:vacc} and \ref{fig:vaccS}, respectively.

\begin{figure}[H]
\begin{minipage}{.5\textwidth}
\begin{center}
\begin{tikzpicture}[->]
\node (s) at (-2,0) {$S$};
\node (i) at (2,0) {$I$};
\node (r) at (0,-2) {$R$};
\draw[blue] (s) to [bend left=20] node[midway,above] {$\beta S I$} (i);
\draw[blue] (i) to [bend left] node[midway,below] {$\alpha I$} (r);
\draw[<-] (r) to [bend left] node[midway,left] {$v(S,I)$} (s);
\end{tikzpicture}
\caption{A simple vaccination model. \label{diag:vacc}}
\end{center}
\end{minipage}
\begin{minipage}{.5\textwidth}
\includegraphics[width=7.7cm, height=4.6cm]{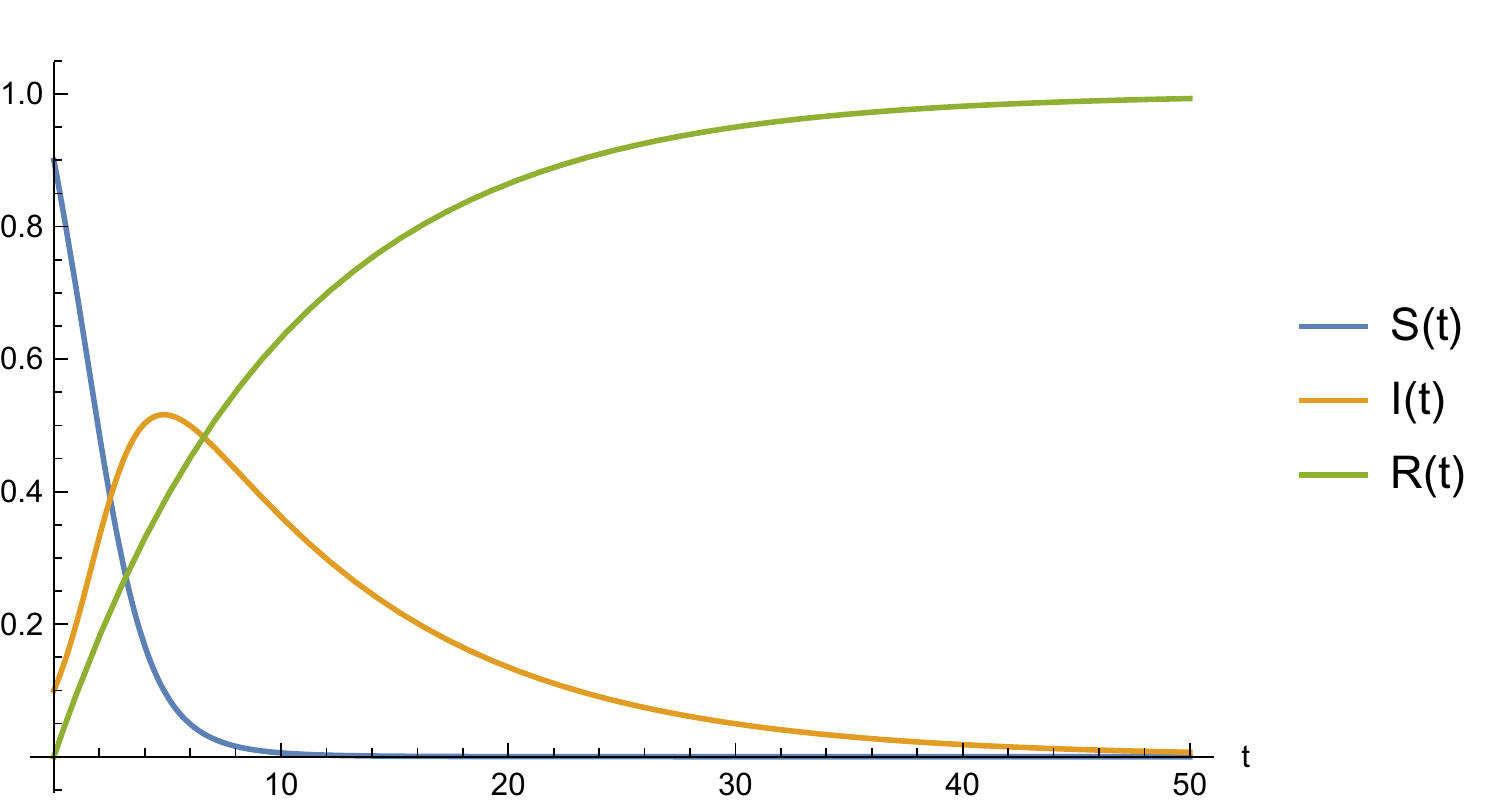} 
\caption{A simple modification of the original SIR with vaccination function $\propto S$ ($\beta=1, \alpha =0.1,v=0.1$). The dynamics are similar to Figure \ref{fig:originalSIR}, but with a smaller infection peak.  \label{fig:vaccS}}
\end{minipage}
\end{figure}

The following choices for the vaccination function will play a prominent role in the following: 
a) $v(S,I) = v I$ and b) $v(S,I) = v S$,
with $v \in \mathbb R_+$.
\hfill$\diamondsuit$
\end{example}

\begin{example}[SIR with vital dynamics]
\label{ex:SIRvitaldyn}
Models that include vital dynamics, i.e. models that take into account the deaths and newborns in the population during the epidemics, do also fit in the generalized SIR system \eqref{eq:generalizedSIR}, provided that the total population remains constant. The dynamical system for such a model is given by
\begin{equation}
\begin{split}
&\dot S = - \beta S I + \underbrace{\textcolor{olive}{(d_S S + d_I I + d_R R)}-\red{d_S S}}_{\varphi_1(S,I)} , \\
&\dot I = \beta S I - \alpha I - \underbrace{\red{d_I I}}_{\varphi_2(S,I)}, \\
&\dot R = \alpha I - \underbrace{\red{d_R R}}_{ \varphi_1(S,I) + \varphi_2(S,I)}, \\
\end{split}
\end{equation}
where $d_S,d_I,d_R$ are the death rates for each compartment, and the born rate is adjusted to keep the population constant. In fact, if $d_S=d_I=d_R=d$, then the newborns would be proportional to the total population $S+I+R$, which is usually the assumption used for these kind of models. This system is thus represented by the diagram
\begin{center}
\begin{tikzpicture}[->]
\node (s) at (-2,0) {$S$};
\node (i) at (2,0) {$I$};
\node (r) at (0,-2) {$R$};
\draw[blue] (s) to [bend left=20] node[midway,above] {$\beta S I$} (i);
\draw[blue] (i) to [bend left] node[midway,below] {$\alpha I$} (r);
\draw[red] (s) to node[midway,below]{$d_S S$} (-4,1);
\draw[red] (i) to node[midway,left]{$d_I I$} (4,1);
\draw[red] (r) to node[midway,right]{$d_R R$} (0,-4);
\draw[olive] (-4,-2) to node[midway,below]{$d_S S + d_I I + d_R R$} (s);
\end{tikzpicture}
\end{center}
In order to rewrite this model in the form  \eqref{eq:poissonSI0} we just have to set
\begin{equation}
\begin{split}
&\varphi_1(S,I) = d_I I + d_R (1-S-I), \\
&\varphi_2(S,I) = -d_I I, \\
\end{split}
\end{equation}
and we also have to restrict the dynamics to $\mathcal{H}=S+I+R=1$. The diagram of the system is given in Figure \ref{diag:vitaldyn} while a possible dynamics is presented in Figure \ref{fig:vitaldyn}.
\hfill$\diamondsuit$

\begin{figure}[H]
\begin{minipage}{.5\textwidth}
\begin{center}
\begin{tikzpicture}[->]
\node (s) at (-2,0) {$S$};
\node (i) at (2,0) {$I$};
\node (r) at (0,-2) {$R$};
\draw[blue] (s) to [bend left=20] node[midway,above] {$\beta S I$} (i);
\draw[blue] (i) to [bend left] node[midway,below] {$\alpha I$} (r);
\draw (r) to [bend left] node[midway,left] {$d_I I + d_R R$} (s);
\draw (r) to [bend left] node[midway,left] {$-d_I I$} (i);
\end{tikzpicture}\caption{A SIR system with vital dynamics. \label{diag:vitaldyn}}
\end{center}
\end{minipage}
\begin{minipage}{.5\textwidth}
\includegraphics[width=7.7cm, height=4.6cm]{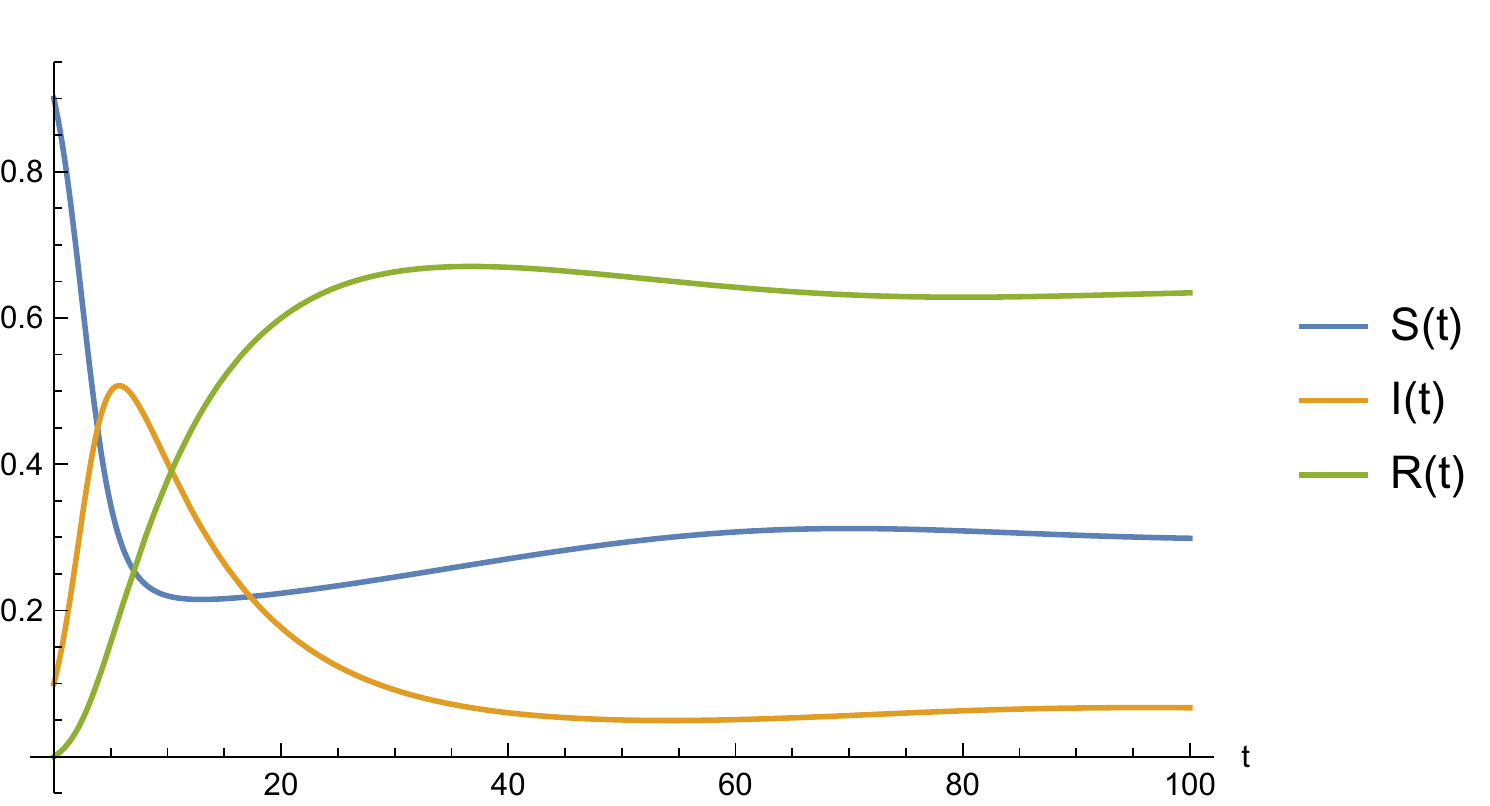} 
\caption{The qualitative behavior of the solutions for this system is similar to the endemic SIRS (see Figure \ref{fig:endemicSIRS}). The inclusion of new susceptible individuals makes the disease to become endemic, and the disease does not disappear in the long term. ($\beta=1, \alpha =0.1,d_S=d_R=0.01, d_I=0.2$). \label{fig:vitaldyn}}
\end{minipage}
\end{figure}
\end{example}

\subsection{A note regarding models with non-constant population}

All the models considered so far have a constant population. However, some models with a non-constant population can be also considered in this framework. Let us set $S(t)+I(t)+R(t)=N(t)$ the total population, which is a non-constant function of time. Thus we have
\begin{equation}
\frac{d}{dt} \left( \frac{S(t)+I(t)+R(t)}{N(t)} \right) = 0 ,
\end{equation}
so let us define 
\begin{equation}
s(t) = \frac{S(t)}{N(t)},\qquad i(t) = \frac{I(t)}{N(t)},\qquad r(t) = \frac{R(t)}{N(t)} .
\end{equation}
If we assume an exponential growth of the total population, we have that 
\begin{equation}
\dot N(t) = (b-d) N(t). 
\end{equation}
For instance this is the case for models including vital dynamics with equal death rates for each compartment. Under this assumption, the dynamics of a generalized SIR model with non-constant population  reads
\begin{equation}
\begin{split}
&\dot S = bN -dS - \beta \frac{S I}{N} + \varphi_1(S,I) , \\
&\dot I = \beta \frac{S I}{N} - \alpha I - d I + \varphi_2(S,I), \\
&\dot R = \alpha I -d R - \varphi_1(S,I) - \varphi_2(S,I). \\
\end{split}
\end{equation}
If $\varphi_1,\varphi_2$ are linear functions (most of the models do consider this assumption), we have that
\begin{equation}
\begin{split}
&\frac{ds}{dt} = \frac{\dot S N - S \dot N}{N^2} = - \beta s i + \underbrace{b-d s + \varphi_1(s,i) - (b-d)s}_{\tilde \varphi_1 (s,i)} , \\
&\frac{di}{dt} = \frac{\dot I N - I \dot N}{N^2} = \beta s i - \alpha i \underbrace{- d i + \varphi_2(s,i) - (b-d)i}_{\tilde \varphi_2 (s,i)} , \\
&\frac{dr}{dt} = \frac{\dot R N - R \dot N}{N^2} = \alpha i \underbrace{-d (1 - s - i) - \varphi_1(s,i) - \varphi_2(s,i) - (b-d)(1-s-i)}_{-\tilde \varphi_1 (s,i) - \tilde \varphi_2 (s,i)}, \\
\end{split}
\end{equation}
where $1-s-i=r$. Clearly the condition
\begin{equation}
\frac{d}{dt} (s+i+r) = 0 , \,
\end{equation}
is thus satisfied. In this way, any generalized SIR model with $\varphi_1,\varphi_2$ linear and with a non-constant population with exponential growth, turns out to be equivalent to a generalized SIR model \eqref{eq:generalSIR} with constant population. Therefore, all the results presented in this paper are also valid for this type of systems.


\section{Hamiltonian structure of compartmental models}
\label{sec:Hamiltoniancompmodels}

In general, being Hamiltonian is quite a restrictive property for a dynamical system. However, we have just proved that the generalized SIR model \eqref{eq:generalizedSIR} is Hamiltonian with no restrictions, and this suggests that a deeper understanding of this result should be possible. In fact, it turns out that this is a particular case of a much more general finding, which essentially states that \emph{any compartmental model with constant population is a Hamiltonian system}.

\begin{proposition}
\label{prop:compartmentalHamiltonian}
Let
\begin{equation}
\dot x^\mu = f^\mu(x^1,\ldots,x^{M-1}), \qquad\qquad \mu \in \{1,\ldots,M\}
\end{equation}
be a system of ODEs, such that 
\begin{equation}
\dot x^1 + \cdots + \dot x^M = 0 .
\end{equation}
This system is a Hamiltonian system, with
\begin{equation}
\mathcal H = \sum_{\mu=1}^M x^\mu,
\end{equation}
and fundamental Poisson brackets given by
\begin{equation}
\begin{split}
&\{x^\mu,x^\nu\}=0, \qquad\qquad \forall \mu,\nu \in \{1,\ldots,M-1\} \\
&\{x^\mu,x^M\}=f^\mu(x^1,\ldots,x^{M-1}), \qquad\qquad \mu \in \{1,\ldots,M-1\} .
\end{split}
\end{equation}

\end{proposition}

\begin{proof}

Clearly, there are only $M-1$ independent equations, since 
\begin{equation}
\dot x^M = -\dot x^1 - \cdots - \dot x^{M-1}.
\end{equation}
Hamilton's equations read
\begin{equation}
\dot x^\mu = \{x^\mu,\mathcal H\}= \{x^\mu,x^M\} = f^\mu(x^1,\ldots,x^{M-1}), \qquad\qquad \forall \mu \in \{1,\ldots,M-1\}
\end{equation}
so the dynamical system is recovered from the Hamiltonian structure, and all Jacobi identities are satisfied since $\{\{\cdot,\cdot\},\cdot\} = 0$, because the subalgebra defined by $x_1,\ldots,x_{M-1}$ is Abelian:
\begin{equation}
\mathrm{Jac}(x^\mu,x^\nu,x^M) = \{\cancel{\{x^\mu,x^\nu\}},x^M\} + \{\underbrace{\{x^M,x^\mu\}}_{-f^\mu(x^1,\ldots,x^{M-1})},x^\nu\}  + \{\underbrace{\{x^\nu,x^M\}}_{f^\nu(x^1,\ldots,x^{M-1})},x^\mu\} = 0 .
\end{equation}
\end{proof}

From this result, and taking into account that the Hamiltonian is  a conserved quantity, we straightforwardly deduce the following

\begin{corollary}
Every compartmental model with constant population is a Hamiltonian system.
\end{corollary}

Moreover, we have that Proposition \ref{prop:generalizedSIRhamiltonian} is just a particular case of the previous Proposition \ref{prop:compartmentalHamiltonian}, with $M=3$, $x^1=S$, $x^2=I$ and $x^3=R$. Moreover, Proposition \ref{prop:compartmentalHamiltonian} shows that the generalized SIR  system \eqref{eq:generalizedSIR} can be naturally written as a Hamiltonian system for two extra Poisson structures, in both cases by keeping the same Hamiltonian. The first Poisson structure corresponds to take $x^3=I$ and reads
\begin{equation}
\label{eq:PoissonSIR_SR}
\begin{split}
&\{ S,I\} = - \beta S (1-S-R) + \varphi_1(S,R), \\
&\{ S,R\} = 0 , \\
&\{ I,R\} = -\alpha (1-S-R) + \varphi_1(S,R) + \varphi_2(S,R) , \\
\end{split}
\end{equation}
while the second one corresponds to $x^3=S$ and is given by
\begin{equation}
\label{eq:PoissonSIR_IR}
\begin{split}
&\{ S,I\} = \alpha I - \beta (1-I-R) I - \varphi_2(I,R) , \\
&\{ S,R\} = -\alpha I + \varphi_1(I,R) + \varphi_2(I,R) , \\
&\{ I,R\} = 0 . \\
\end{split}
\end{equation}
The existence of different Poisson structures will be relevant in the following Section, since each of them will allow us to define different interacting SIR models.

\subsection{Generalized SEIR model}

As an application of the result presented above, namely that every compartmental epidemiological model with constant population is a Hamiltonian system, we present the Hamiltonian structure for the SEIR system \cite{Li1999}. This model differs from SIR in the existence of an extra compartment, $E$, representing the individuals that have been exposed to the infection but are not infective yet. Therefore, this model accounts for infections that present an incubation period. As we did for the SIR model, we can define a family of \emph{generalized SEIR systems}, defined by the following system of ODEs
\begin{equation}
\begin{split}
\label{eq:generalSEIR}
&\dot S = - \beta S I + \varphi_1(S,E,I) , \\
&\dot E = \beta S I - \epsilon E + \varphi_2(S,E,I) , \\
&\dot I = - \alpha I + \epsilon E + \varphi_3(S,E,I), \\
&\dot R = \alpha I - \varphi_1(S,I) - \varphi_2(S,I) - \varphi_3(S,E,I), \\
\end{split}
\end{equation}
where $\varphi_i(S,E,I)$ are arbitrary functions of $S,E,I$. This system is represented by the diagram
\begin{center}
\begin{tikzpicture}[->]
\node (s) at (-3,0) {$S$};
\node (e) at (3,0) {$E$};
\node (i) at (3,-3) {$I$};
\node (r) at (-3,-3) {$R$};
\draw[blue] (s) to [bend left=20] node[midway,above] {$\beta S I$} (e);
\draw[blue] (e) to [bend left] node[midway,right] {$\epsilon E$} (i);
\draw[blue] (i) to [bend left] node[midway,below] {$\alpha I$} (r);
\draw (r) to [bend left] node[midway,left] {$\varphi_1(S,E,I)$} (s);
\draw (r) to [bend left] node[midway,above] {$\varphi_2(S,E,I)$} (e);
\draw (r) to [bend left] node[midway,above] {$\varphi_3(S,E,I)$} (i);
\end{tikzpicture}
\end{center}
Of course the original SEIR is recovered when $\varphi_i(S,E,I)=0$. The analogue of Examples \ref{ex:SIRSendemic}-\ref{ex:SIRvitaldyn} can be defined in a similar way. 

The fundamental Poisson brackets that endow the generalized SEIR model with a Hamiltonian structure are
\begin{equation}
\begin{split}
\label{eq:poissonSEIR}
&\{ S,R\} = - \beta S I + \varphi_1(S,E,I) , \\
&\{ I,R\} = -\alpha I + \epsilon E + \varphi_3(S,E,I) , \\
&\{ E,R\} = \beta S I - \epsilon E + \varphi_2(S,E,I). \\
&\{ S,I\} = 0, \qquad\qquad \{ E,S\} = 0, \qquad\qquad \{ E,I\} = 0, \\
\end{split}
\end{equation}
which is a particular case of Proposition \ref{prop:compartmentalHamiltonian} with $M=4$ and $x^M=R$. Again, three more different Poisson structures giving rise to the same dynamical system are possible by setting $x^M=S$, $x^M=E$ or $x^M=I$. In all the cases the Hamiltonian is just the total population
\begin{equation}
\mathcal H = S + E + I + R ,
\end{equation}
and it is straightforward to show that Hamilton's equation are just \eqref{eq:generalSEIR}. Similarly to the SIR case, given that the dynamics take place on $\mathcal H^{-1}(1)$, arbitrary functions of $R$ are indeed included in \eqref{eq:generalSEIR}, and therefore this is the most general 4-dimensional compartmental model with constant population containing the original SEIR system.

\section{Models for $N$ interacting SIR populations}
\label{sec:interactingSIR}

The aim of this Section is to show that we can define an epidemiological Hamiltonian system providing the dynamics of  $N$ populations $\mathcal P_1,\ldots,\mathcal P_N$, each of them presenting the `internal' dynamics of a generalized SIR model \eqref{eq:generalizedSIR}, and such that the N populations interact through an arbitrary transference from  one compartment of $\mathcal P_a$ to the {\em same} compartment of $\mathcal P_b$. We stress that the model will allow for the transference function to be an arbitrary function $\tau_{a b}(S_a,I_a,S_b,I_b)$ of the variables of the populations involved. For the sake of brevity, we only present the results for the case in which the interacting compartments are $R_a$ (recovered individuals) which is the model constructed from the Hamiltonian description \eqref{prop:generalizedSIRhamiltonian}. However, the brackets \eqref{eq:PoissonSIR_SR} and \eqref{eq:PoissonSIR_IR} allow to construct Hamiltonian interacting systems with interacting compartments $I_a$ and $S_a$, respectively.

\begin{proposition}
\label{pr:interacting_generalized_SIR}
Let us consider the Hamiltonian system defined by the 3N-dimensional Poisson algebra with fundamental Poisson brackets 
\begin{equation}
\begin{split}
\label{eq:poissonSI0prop}
&\{ I_a,R_a\} = -\alpha_a I_a + \beta_a S_a I_a + \varphi_{a,2}(S_a,I_a) , \\
&\{ S_a,R_a\} = - \beta_a S_a I_a + \varphi_{a,1}(S_a,I_a) , \\
&\{ S_a,I_a\} = 0, \\
&\{ R_a,R_b \} = -\tau_{a b}(S_a,I_a,S_b,I_b) , \\
\end{split}
\end{equation}
for $a,b \in\{1,\ldots, N\}$ and $b \neq a$, while the remaining Poisson brackets vanish. The Hamiltonian function
\begin{equation}
\label{eq:H_SIRprop}
\mathcal H = \sum_{k=1}^N \mathcal H_k(S_k,I_k,R_k) = \sum_{k=1}^N (S_k + I_k + R_k) \, ,
\end{equation}
defines a dynamical system with $N$ interacting populations $\mathcal P_1, \ldots, \mathcal P_N$, each of them evolving according to the dynamics of a generalized SIR model \eqref{eq:generalSIR}. The generalized SIR dynamics for each population can be arbitrary and, in particular, different for each population. Each pair of populations $\mathcal P_a$ and $\mathcal P_b$ interact by means of the transference of recovered individuals from $R_a$ to $R_b$ given by $\tau_{a b}(S_a,I_a,S_b,I_b)$. 
\end{proposition}

\begin{proof}
Hamilton's equations are computed from the Hamiltonian function and the Poisson brackets:
\begin{equation}
\label{eq:generalizedSIR_Npop}
\begin{split}
\dot S_a &= \{S_a,\mathcal H\} = \sum_{k=1}^N \left(\cancel{\{S_a,S_k\} }+ \cancel{\{S_a,I_k\}} + \{S_a,R_k\} \right) = \{S_a,R_a\} = - \beta_a S_a I_a + \varphi_{a,1}(S_a,I_a) ,\\
\dot I_a &= \{I_a,\mathcal H\} = \sum_{k=1}^N \left(\cancel{\{I_a,S_k\} }+ \cancel{\{I_a,I_k\}} + \{I_a,R_k\} \right) = \{I_a,R_a\} =\beta_a S_a I_a - \alpha_a I_a + \varphi_{a,2}(S_a,I_a), \\
\dot R_a &= \{R_a,\mathcal H\} = \sum_{k=1}^N \left(\{R_a,S_k\} + \{R_a,I_k\} + \{R_a,R_k\} \right) = \{R_a,S_a\} + \{R_a,I_a\} + \sum_{k=1}^N \{R_a,R_k\} \\
&=\alpha_a I_a - \varphi_{a,1}(S_a,I_a) - \varphi_{a,2}(S_a,I_a) - \sum_{k\neq a}^N \tau_{a k}(S_a,I_a,S_k,I_k). \\
\end{split}
\end{equation}
Therefore, the dynamics of each population $\mathcal P_a$ is that of a generalized SIR model, with arbitrary $\alpha_a$, $\beta_a$, $ \varphi_{a,1}(S_a,I_a)$ and $\varphi_{a,2}(S_a,I_a)$ for each of them. Moreover, the population $\mathcal P_a$ interact with the population $\mathcal P_b$ by means of the transfer function $\tau_{a b}(S_a,I_a,S_b,I_b)$ between $R_a$ and $R_b$.

The proof that \eqref{eq:poissonSI0prop} indeed defines a Poisson algebra is a direct consequence of the fact that $\{f_1,f_2\} \subset \mathcal C^\infty(S_1,I_1,\ldots, S_N,I_N)$ for any smooth functions $f_1,f_2$, and $\mathcal C^\infty(S_1,I_1,\ldots, S_N,I_N)$ is a Poisson commutative subalgebra. 
\end{proof}

Note that the total population for this system is indeed kept constant under the chosen interactions, since
\begin{equation}
\dot S + \dot I + \dot R = \sum_{a=1}^N \left( - \sum_{k\neq a}^N \tau_{a k}(S_a,I_a,S_k,I_k) \right) = 0,
\end{equation}
where we have used the skew-symmetry of the set of functions $\tau_{a b}(S_a,I_a,S_b,I_b)$ under the exchange $a \leftrightarrow b$. For example, taking $N=3$ in the previous Proposition we obtain the system represented by the following diagram
\begin{center}
\begin{tikzpicture}[->]
\node (s2) at (-5,4) {$S_2$};
\node (i2) at (-6.8,0.5) {$I_2$};
\node (r2) at (-3,1) {$R_2$};
\draw[blue] (s2) to [bend right=20] node[midway] {$\beta_2 S_2 I_2$} (i2);
\draw[blue] (i2) to [bend right] node[midway,below] {$\alpha_2 I_2$} (r2);
\draw (r2) to [bend right] node[midway,right] {$\varphi_{2,1}(S_2,I_2)$} (s2);
\draw (r2) to [bend right] node[midway,below] {$\varphi_{2,2}(S_2,I_2)$} (i2);

\node (s3) at (5,4) {$S_3$};
\node (i3) at (6.8,0.5) {$I_3$};
\node (r3) at (3,1) {$R_3$};
\draw[blue] (s3) to [bend left=20] node[midway] {$\beta_3 S_3 I_3$} (i3);
\draw[blue] (i3) to [bend left] node[midway,below] {$\alpha_3 I_3$} (r3);
\draw (r3) to [bend left] node[midway,left] {$\varphi_{3,1}(S_3,I_3)$} (s3);
\draw (r3) to [bend left] node[midway,below] {$\varphi_{3,2}(S_3,I_3)$} (i3);

\node (s1) at (-2,-5) {$S_1$};
\node (i1) at (2,-5) {$I_1$};
\node (r1) at (0,-2) {$R_1$};
\draw[blue] (s1) to [bend right=20] node[midway,below] {$\beta_1 S_1 I_1$} (i1);
\draw[blue] (i1) to [bend right] node[midway,right] {$\alpha_1 I_1$} (r1);
\draw (r1) to [bend right] node[midway,left] {$\varphi_{1,1}(S_1,I_1)$} (s1);
\draw (r1) to [bend right] node[midway,above] {$\varphi_{1,2}(S_1,I_1)$} (i1);

\draw[purple] (r1) to [bend right=30] node[midway,right] {$\tau_{1 3}(S_1,I_1,S_3,I_3)$} (r3);
\draw[purple] (r3) to [bend right=30] node[midway,above] {$\tau_{3 2}(S_3,I_3,S_2,I_2)$} (r2);
\draw[purple] (r2) to [bend right=30] node[midway,left] {$\tau_{2 1}(S_2,I_2,S_1,I_1)$} (r1);

\end{tikzpicture}
\end{center}

An example of the dynamics for this kind of interacting systems is presented in Figure \ref{fig:3pop}. 

\begin{figure}[t]
\begin{minipage}{.5\textwidth}
\begin{center}
\includegraphics[width=7.7cm, height=4.6cm]{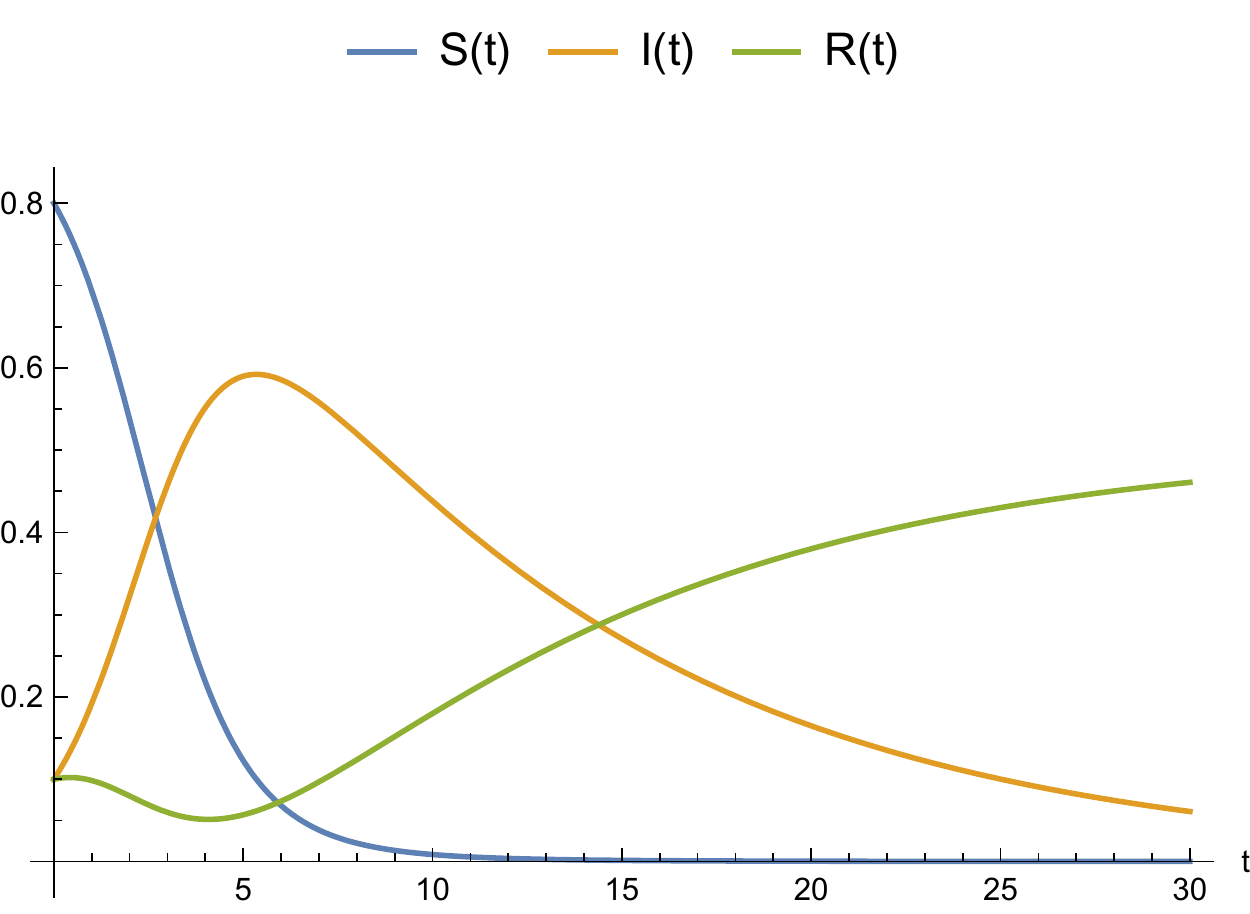} \\
\includegraphics[width=7.7cm, height=4.6cm]{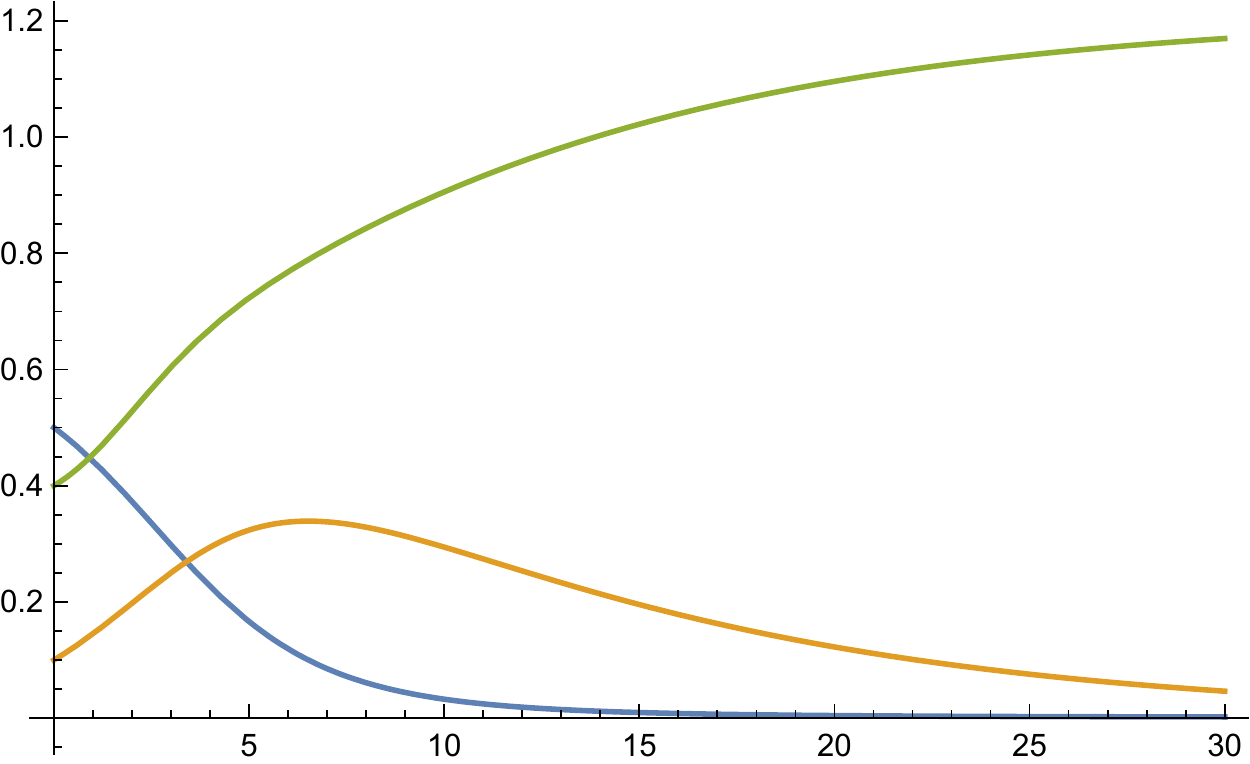}
\end{center}
\end{minipage}
\begin{minipage}{.5\textwidth}
\includegraphics[width=7.7cm, height=4.6cm]{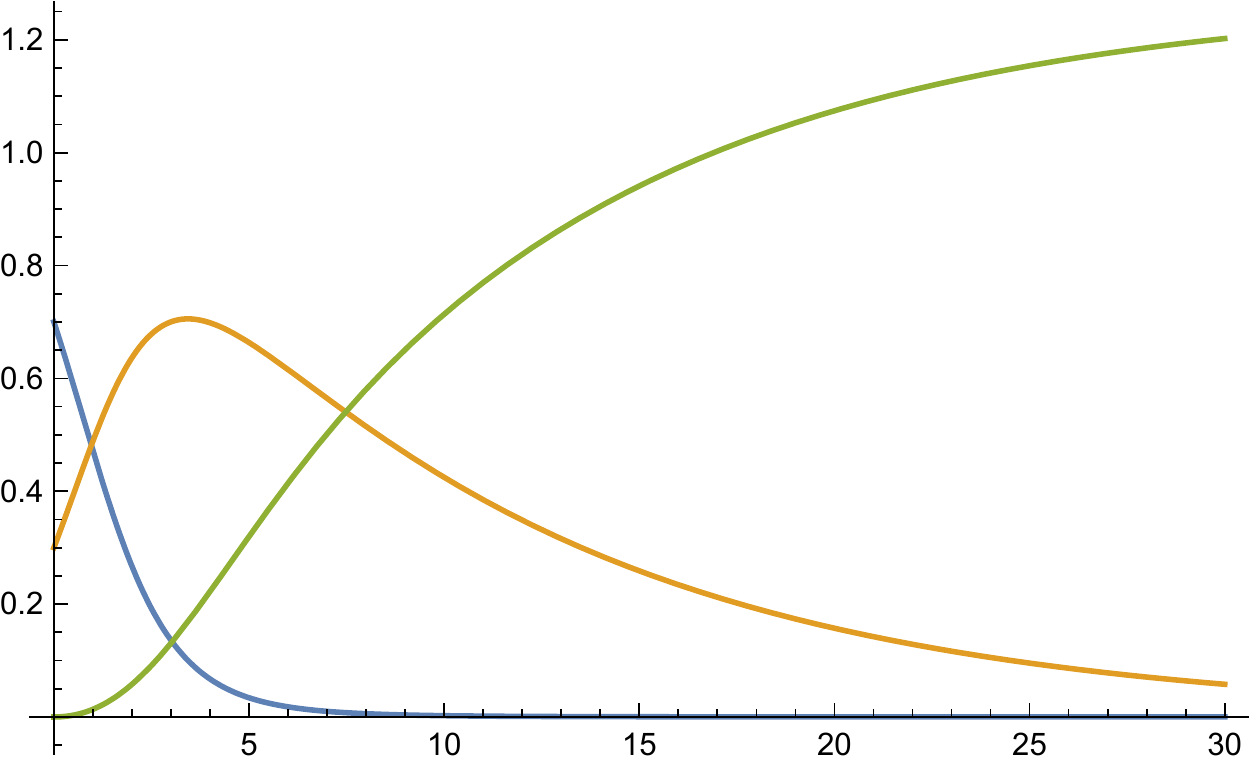} \\
\includegraphics[width=7.7cm, height=4.6cm]{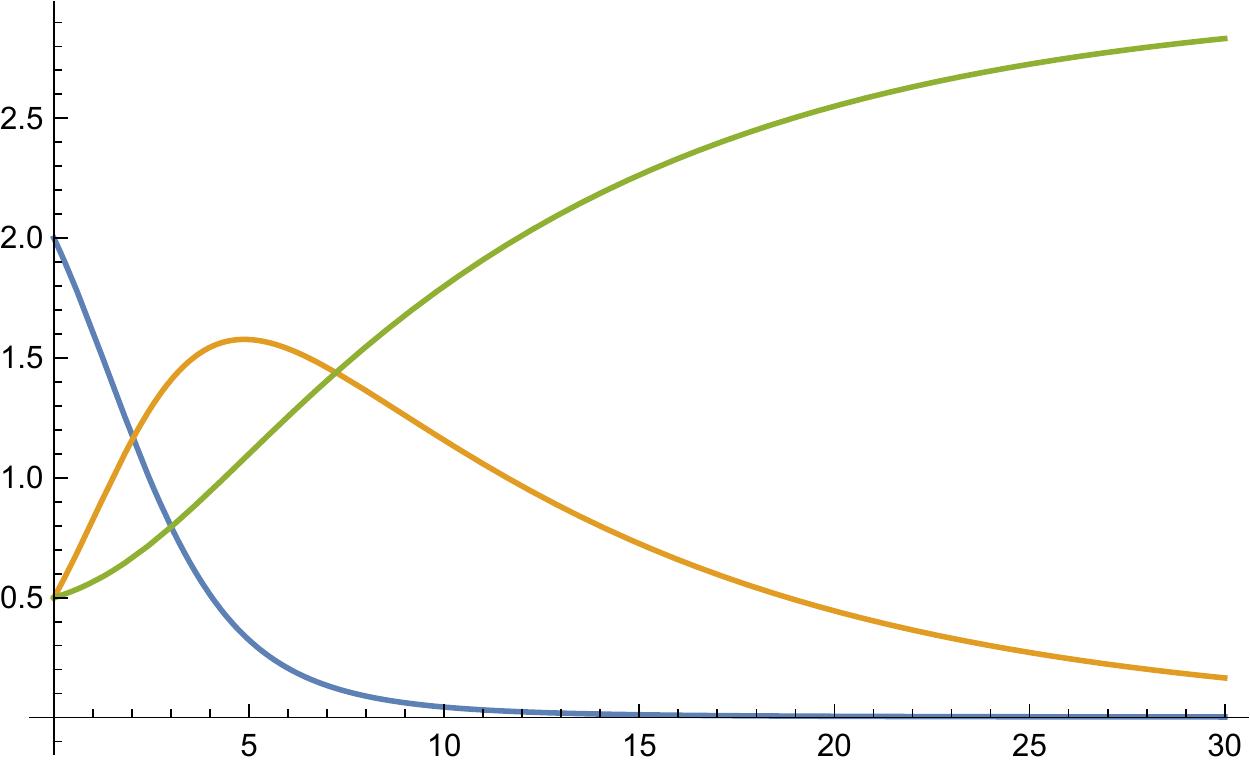} 
\end{minipage}
\caption{Dynamics of an interacting system with 3 populations, with interaction function $\tau_{ab} = \kappa (S_a +S_b+I_a-I_b)$. 
Top-left: First population. Top-right: Second population. Bottom-left: Third population. Bottom-right: Total population. ($\beta_a = 1, \alpha_a =0.1, \kappa=0.1, S_1(0)=0.8, I_1(0)=0.1, R_1(0)=0.1, S_2(0)=0.7, I_2(0)=0.3, R_2(0)=0, S_3(0)=0.5, I_3(0)=0.1, R_3(0)=0.4$).
\label{fig:3pop}}
\end{figure}

\section{$N$ interacting compartmental models}
\label{sec:interacting_compartmental}

The previous Section presents the case of $N$ interacting generalized SIR systems. We recall that in Section \ref{sec:Hamiltoniancompmodels} we have proved that any compartmental model with constant population can be also written as a Hamiltonian system in a similar manner. Therefore, it seems natural to wonder whether we can define interacting Hamiltonian systems based on more general epidemiological models. This is indeed the case, if the dynamical system defining the dynamics of each population depends only on $M-1$ out of the $M$ compartments, in the form described in the following

\begin{proposition}

Consider a system of $M \times N$ ODEs given by 
\begin{equation}
\label{eq:M_N_ODEsystem}
\begin{split}
&\dot x^\mu_a = f^\mu_a(x^1_a,\ldots,x^{M-1}_a), \qquad\qquad \forall \mu \in \{1,\ldots,M-1\}, \quad a \in \{1,\ldots,N\} \\
&\dot x^M_a = - \sum_{\mu = 1}^{M-1}  f^\mu_a(x^1_a,\ldots,x^{M-1}_a) - \sum_{b=1}^N \tau_{ab}(x^1_a,\ldots,x^{M-1}_a,x^1_b,\ldots,x^{M-1}_b) , \qquad \forall a \in \{1,\ldots,N\} , \quad b \neq a. \\
\end{split}
\end{equation}
This system is Hamiltonian with respect to the Poisson structure
\begin{equation}
\begin{split}
&\{x^\mu_a,x^\nu_a\}=0 , \\
&\{x^\mu_a,x^M_a\}=f^\mu_a(x^1_a,\ldots,x^{M-1}_a) , \\
&\{x_a^M,x^M_b\}= -\tau_{ab} (x^1_a,\ldots,x^{M-1}_a,x^1_b,\ldots,x^{M-1}_b) , \\
\end{split}
\end{equation}
and the Hamiltonian function
\begin{equation}
\mathcal H = \sum_{a=1}^N \sum_{\mu=1}^M x^\mu_a .
\end{equation}

\end{proposition}

\begin{proof}

Hamilton's equations read
\begin{equation}
\begin{split}
\dot x^\mu_a&=\{x^\mu_a,\mathcal H \} = \sum_{b=1}^N \sum_{\nu=1}^M \{x^\mu_a, x^\nu_b\} = \sum_{b=1}^N \sum_{\nu=1}^M \delta_{ab} \delta^{\nu M} f^\mu_a(x^1_a,\ldots,x^{M-1}_a) = f^\mu_a(x^1_a,\ldots,x^{M-1}_a), \qquad \mu \neq M, \\
\dot x^M_a&=\{x^M_a,\mathcal H \} = \sum_{b=1}^N \sum_{\nu=1}^M \{x^M_a, x^\nu_b\} \\ 
&=\sum_{b=1}^N \sum_{\nu=1}^M \left( -\delta_{ab} f^\nu_a(x^1_a,\ldots,x^{M-1}_a) - \delta^{M\nu} \tau_{ab}(x^1_a,\ldots,x^{M-1}_a,x^1_b,\ldots,x^{M-1}_b) \right) \\
&=-\sum_{\nu=1}^{M-1} f^\nu_a(x^1_a,\ldots,x^{M-1}_a) - \sum_{b=1}^N \tau_{ab}(x^1_a,\ldots,x^{M-1}_a,x^1_b,\ldots,x^{M-1}_b) ,
\end{split}
\end{equation}
which coincides with \eqref{eq:M_N_ODEsystem}. The fact that $\{x^1_a,\ldots,x^{M-1}_a\}$ are Abelian Poisson subalgebras for all $a \in \{1,\ldots,N\}$ implies that all the Jacobi identities are trivially satisfied, since the only non-trivial Poisson brackets are contained in some product of this commutative subalgebras, and therefore $\{\{g,g'\},g''\}=0$ for any functions $g,g',g''$.
\end{proof}

Note that we have assumed that the functions defining the dynamical system do not depend of $x^M$. This is similar to the assumption in Proposition \ref{prop:compartmentalHamiltonian}, but there we argued that the system is the most general for models with constant population, since we can restrict the dynamics to $\mathcal H^{-1}(1)$ and we could write $x^M=1-x^1-\ldots -x^{M-1}$. This is no longer true for interacting systems, since the population of each $\mathcal P_a$ is not conserved. However, note that the total population for the complete system is indeed conserved, since
\begin{equation}
\sum_{a=1}^N \sum_{\mu=1}^M \dot x^\mu_a = \sum_{a=1}^N \sum_{b=1}^N \tau_{ab}(x^1_a,\ldots,x^{M-1}_a,x^1_b,\ldots,x^{M-1}_b) = 0 ,
\end{equation}
due to the skew-symmetry of $\tau_{ab}(x^1_a,\ldots,x^{M-1}_a,x^1_b,\ldots,x^{M-1}_b)$ under the exchange $a \leftrightarrow b$.

\section{Bi-Hamiltonian structure of epidemic models}
\label{sec:biHamiltonian}

As it was mentioned in the Introduction, the original SIR system is bi-Hamiltonian \cite{Nutku1990sir}. In fact, some specific generalizations of this system were proved to be bi-Hamiltonian in \cite{Gumral1993}. In this Section we show that the systems given in Examples \ref{ex:SIRSendemic}-\ref{ex:SIRvaccination} are bi-Hamiltonian. Moreover, by coupling $N$ of those bi-Hamiltonian systems we obtain different families of $3N$-dimensional bi-Hamiltonian systems.

\subsection{Endemic SIRS}

\begin{proposition}
\label{pr:endemicSIRSbiham}
The system given in Example \ref{ex:SIRSendemic}, with $\alpha$ and $\beta$ constants, is bi-Hamiltonian with respect to the following two compatible Poisson structures and Hamiltonian functions:
\begin{fleqn}
\begin{equation}
\begin{cases}
\text{Hamiltonian: } \mathcal H_{1}=S+I+R, \vspace{0.2cm}\\
\text{Poisson structure: } \{S,I\}_{1}=0, \quad \{S,R\}_{1}=-\beta S\,I+\mu\,I, \quad \{I,R\}_{1}=\beta\,S\,I-(\alpha+\mu)\,I.
\end{cases}
\end{equation}
\end{fleqn}
\begin{fleqn}
\begin{equation}
\begin{cases}
\text{Hamiltonian: } \mathcal H_{2}=-\left(R+\frac{\alpha}{\beta}\log(\beta S - \mu) \right), \vspace{0.2cm}\\
\text{Poisson structure: } \{S,I\}_{2}=-\beta\,S\,I+\mu\,I, \quad \{S,R\}_{2}=\beta S\,I-\mu\,I, \quad \{I,R\}_{2}=-\beta\,S\,I + \mu\,I.
\end{cases}
\end{equation}
\end{fleqn}
\end{proposition}

\begin{proof}
The dynamical system \eqref{eq:SIRSendemic} is obtained as
\begin{equation}
\begin{split}
&\dot S = \{S,\mathcal H_1 \}_1 = \{S,\mathcal H_2 \}_2 = - \beta S I + \mu\,I , \\
&\dot I = \{I,\mathcal H_1 \}_1 = \{I,\mathcal H_2 \}_2 = \beta S I - (\alpha+\mu) I, \\
&\dot R = \{R,\mathcal H_1 \}_1 = \{R,\mathcal H_2 \}_2 = \alpha I . \\
\end{split}
\end{equation}
Also it is easy to check that $\{ f,\mathcal H_{1} \}_2 =0$, so $\mathcal H_{1}$ is a Casimir function for $\{\cdot,\cdot \}_2$. Obviously $\{\cdot,\cdot \}_1$ is a Poisson structure because it is a particular case of \eqref{eq:poissonSI0}. The Jacobi identity for the brackets $\{\cdot,\cdot \}_2$ is fulfilled
\begin{equation}
\mathrm{Jac}_2 (S,I,R) = \{\{S,I\}_2,R\}_2 + \{\{R,S\}_2,I\}_2 + \{\{I,R\}_2,S\}_2 = \{-\beta\,S\,I+\mu\,I, \mathcal H_{1} \}_2 = 0 ,\, 
\end{equation}
since $\mathcal H_{1}$ is a Casimir function for $\{\cdot,\cdot \}_2$. Finally, the two Poisson structures are compatible since the brackets 
\begin{equation}
\begin{split}
&\{S,I\} = \{S,I\}_1 + \{S,I\}_2 = - \beta S I + \mu\,I, \\ 
&\{S,R\} = \{S,R\}_1 + \{S,R\}_2 = 0, \\ 
&\{I,R\} = \{I,R\}_1 + \{I,R\}_2 = - \alpha I, \\
\end{split}
\end{equation}
again define a Poisson structure with vanishing Jacobi identity. In fact, this is the generalization to $\mu \neq 0$ of the Poisson structure considered in \cite{Nutku1990sir}.
\end{proof}

Several bi-Hamiltonian endemic SIRS systems as the ones considered in the previous Proposition can be coupled to each other by keeping a Hamiltonian description similar to the one given by Proposition \ref{pr:interacting_generalized_SIR}. It is therefore natural to wonder whether this description maintains the bi-Hamiltonian structure for the coupled system. This is indeed the case, as we prove in the following

\begin{proposition}
\label{prop:SIRSendemic_N_biham}
The particular case of the Hamiltonian system defined in Proposition \ref{pr:interacting_generalized_SIR}, with $\alpha_a$ and $\beta_a$ constants, $\varphi_{a,1}=\mu_a I_a$, $\varphi_{a,2}=-\mu_a I_a$ and $\tau_{a b} (S_a,I_a,S_b,I_b)= \tau_{a b}$, is bi-Hamiltonian, with the second Poisson structure given by the non-zero fundamental Poisson brackets
\begin{equation}
\begin{split}
&\{S_a,I_a\}_{2}=-\beta_a\,S_a\,I_a+\mu_a\,I_a, \\ 
&\{S_a,R_a\}_{2}=\beta_a S_a\,I_a-\mu_a\,I_a, \\ 
&\{I_a,R_a\}_{2}=-\beta_a \,S_a\,I_a + \mu_a\,I_a, \\
&\{R_a,R_b\}_{2}= \tau_{a b},
\end{split}
\end{equation}
and the Hamiltonian being
\begin{equation}
\mathcal H_{2} =- \sum_{k=1}^N \left(R_k+\frac{\alpha_k}{\beta_k}\log(\beta_k S_k - \mu_k) \right) .
\end{equation}
\end{proposition}

\begin{proof}

The system defined by the second Hamiltonian structure is given by  
\begin{equation}
\begin{split}
&\dot{S}_a = \{S_a,\mathcal H_{2} \}_2 = - \beta_a S_a I_a + \mu_a I_a, \\
&\dot{I}_a = \{I_a,\mathcal H_{2} \}_2 = \beta_a S_a I_a -(\alpha_a + \mu_a) I_a, \\
&\dot{R}_a = \{R_a,\mathcal H_{2} \}_2 = \alpha_a I_a - \sum_{k\neq a}^N \tau_{a k}, \\
\end{split}
\end{equation}
which coincides with \eqref{eq:generalizedSIR_Npop} setting $\varphi_{a,1}=\mu_a I_a$, $\varphi_{a,2}=-\mu_a I_a$ and $\tau_{a b} (S_a,I_a,S_b,I_b)= \tau_{a b}$.

The proof that $\{\cdot,\cdot\}_2$ is a Poisson structure is straightforward since the $N$ Jacobi identities involving variables from a single population are satisfied by Proposition \ref{pr:endemicSIRSbiham}, while the rest of the Jacobi identities are fulfilled just by noting that 
$\{\{X,Y\},Z\}=0$ whenever $X,Y,Z$ do not belong to the same population, due to the fact that the only non-vanishing bracket among different populations is a constant function.

Finally, the compatibility between $\{\cdot,\cdot\}_1$ and $\{\cdot,\cdot\}_2$ is proven by computing the Jacobi identities for
\begin{equation}
\begin{split}
&\{S_a,I_a\} = \{S_a,I_a\}_1 + \{S_a,I_a\}_2 = - \beta_a S_a I_a + \mu_a \,I_a, \\ 
&\{S_a,R_a\} = \{S_a,R_a\}_1 + \{S_a,R_a\}_2 = 0, \\ 
&\{I_a,R_a\} = \{I_a,R_a\}_1 + \{I_a,R_a\}_2 = - \alpha_a I_a, \\
\end{split}
\end{equation}
which was already proven in Proposition \ref{pr:endemicSIRSbiham}.  
\end{proof}

\subsection{A SIR model with a vaccination function}

Similar results to those of Propositions \ref{pr:endemicSIRSbiham} and \ref{prop:SIRSendemic_N_biham} can be obtained for the system from Example \ref{ex:SIRvaccination} a) and b), corresponding to a SIR system in which the susceptible population is vaccinated with a rate proportional to $I$ and $S$, respectively. In particular, these systems are bi-Hamiltonian, and $N$  such systems can be coupled by preserving a bi-Hamiltonian structure for the corresponding $3N$-dimensional system. As the proofs of these results are similar to those of Propositions \ref{pr:endemicSIRSbiham} and \ref{prop:SIRSendemic_N_biham} we omit them for the sake of brevity.

\begin{proposition}
\label{pr:SIRvaccI_biham}
The system given in Example \ref{ex:SIRvaccination} a), with $\alpha$ and $\beta$ constants and setting $v(S,I)=v I$, is bi-Hamiltonian with respect to the following two compatible Poisson structures and Hamiltonian functions:
\begin{fleqn}
\begin{equation}
\begin{cases}
\text{Hamiltonian: } \mathcal H_{1}=S+I+R, \vspace{0.2cm}\\
\text{Poisson structure: } \{S,I\}_{1}=0, \quad \{S,R\}_{1}=-\beta S\,I- v\,I, \quad \{I,R\}_{1}=\beta\,S\,I-\alpha\,I.
\end{cases}
\end{equation}
\end{fleqn}
\begin{fleqn}
\begin{equation}
\begin{cases}
\text{Hamiltonian: } \mathcal H_{2}=-\left( R+\frac{(\alpha+v)}{\beta}\log(\beta S + v) \right), \vspace{0.2cm}\\
\text{Poisson structure: } \{S,I\}_{2}=-\beta\,S\,I-v\,I, \quad \{S,R\}_{2}=\beta S\,I+v\,I, \quad \{I,R\}_{2}=-\beta\,S\,I -v\,I.
\end{cases}
\end{equation}
\end{fleqn}
\end{proposition}

\begin{proposition}
\label{prop:vacc_I_N_biham}
The particular case of the Hamiltonian system defined in Proposition \ref{pr:interacting_generalized_SIR}, with $\alpha_a$ and $\beta_a$ constants, $\varphi_{a,1}=-v_a I_a$, $\varphi_{a,2}=0$ and $\tau_{a b} (S_a,I_a,S_b,I_b)= \tau_{a b}$, is bi-Hamiltonian, with the second Poisson structure defined by the non-zero fundamental Poisson brackets
\begin{equation}
\begin{split}
&\{S_a,I_a\}_{2}=-\beta_a\,S_a\,I_a-v_a\,I_a, \\ 
&\{S_a,R_a\}_{2}=\beta_a S_a\,I_a+v_a\,I_a, \\ 
&\{I_a,R_a\}_{2}=-\beta\,S_a\,I_a -v_a\,I_a, \\
&\{R_a,R_b\}_{2}= \tau_{a b},
\end{split}
\end{equation}
and the Hamiltonian 
\begin{equation}
\mathcal H_{2} =- \sum_{k=1}^N \left(R_k+\frac{(\alpha_k+v_k)}{\beta_k} \log(\beta_k S_k +v_k) \right).
\end{equation}
\end{proposition}

\begin{proposition}
\label{pr:SIRvaccS_biham}
The system given in Example \ref{ex:SIRvaccination} b), with $\alpha$ and $\beta$ constants, setting $v(S,I)=v S$, is bi-Hamiltonian with respect to the following two compatible Poisson structures and Hamiltonian functions:
\begin{fleqn}
\begin{equation}
\begin{cases}
\text{Hamiltonian: } \mathcal H_{1}=S+I+R, \vspace{0.2cm}\\
\text{Poisson structure: } \{S,I\}_{1}=0, \quad \{S,R\}_{1}=-\beta S\,I- v\,S, \quad \{I,R\}_{1}=\beta\,S\,I-\alpha\,I.
\end{cases}
\end{equation}
\end{fleqn}
\begin{fleqn}
\begin{equation}
\begin{cases}
\text{Hamiltonian: } \mathcal H_{2}=-\left(R+\frac{\alpha}{\beta}\log S - \frac{v}{\beta}\log I \right) = -\left(R+\frac{1}{\beta}\log (S^\alpha I^{-v}) \right), \vspace{0.2cm}\\
\text{Poisson structure: } \{S,I\}_{2}=-\beta\,S\,I, \quad \{S,R\}_{2}=\beta S\,I, \quad \{I,R\}_{2}=-\beta\,S\,I.
\end{cases}
\end{equation}
\end{fleqn}
\end{proposition}

\begin{proposition}
\label{prop:vacc_I_N_biham}
The particular case of the Hamiltonian system defined in Proposition \ref{pr:interacting_generalized_SIR}, with $\alpha_a$ and $\beta_a$ constants, $\varphi_{a,1}=-v_a S_a$, $\varphi_{a,2}=0$ and $\tau_{a b} (S_a,I_a,S_b,I_b)= \tau_{a b}$, is bi-Hamiltonian, with the second Poisson structure defined by the non-zero fundamental Poisson brackets
\begin{equation}
\begin{split}
&\{S_a,I_a\}_{2}=-\beta_a\,S_a\,I_a, \\ 
&\{S_a,R_a\}_{2}=\beta_a S_a\,I_a, \\ 
&\{I_a,R_a\}_{2}=-\beta\,S_a\,I_a, \\
&\{R_a,R_b\}_{2}= \tau_{a b},
\end{split}
\end{equation}
and the Hamiltonian 
\begin{equation}
\mathcal H_{2} =- \sum_{k=1}^N \left(R_k+\frac{\alpha_k}{\beta_k} \log S_k - \frac{v_k}{\beta_k} \log I_k \right).
\end{equation}
\end{proposition}

Note that in Propositions \ref{pr:endemicSIRSbiham}, \ref{pr:SIRvaccI_biham} and \ref{pr:SIRvaccS_biham} the second Hamiltonian $\mathcal H_2$ is not a Casimir for the first Poisson algebra $\{\cdot,\cdot\}_1$. However, in each case the Casimir of this first Poisson algebra is obtained from $\mathcal H_2$ by substituting $R \to S+I$, so these two functions indeed coincide on the hypersurface of constant population. Obviously, this Casimir could have been taken equivalently as the second Hamiltonian function. However, this is no longer true for the interacting case.


\section{Exact analytical solutions from Casimir functions}
\label{sec:exactsol}

Although the exact integrability of different epidemiological models has attracted some attention (see for instance \cite{Nucci2004,LA2004}), only recently the exact analytical solution of the original SIR system has been found \cite{Miller2012,Harko2014}. In this Section we show that  the Hamiltonian structure is useful in order to find exact analytical solutions for the systems from Examples \ref{ex:SIRSendemic} -- \ref{ex:SIRvaccination}, which are generalizations of the original SIR system. In particular, we use the fact that any 3-dimensional Poisson manifold has one Casimir function in order to reduce the 3-dimensional dymensional system to a 1-dimensional problem, which  can be always integrated by quadratures. 

For all the following systems, we solve the differential equations for the initial conditions $S(0)=S_0$, $I(0) = 1-S_0$ and $R(0)=0$. This is equivalent to consider the dynamics restricted to the symplectic leaf given by the value $\mathcal C = 0$ of the Casimir function.

\begin{example}[Analytical solution of the SIRS endemic model]

Recall from Proposition \ref{pr:endemicSIRSbiham} that a Casimir function for the system of Example \ref{ex:SIRSendemic} is given by 
\begin{equation}
\mathcal C = -\left(R+\frac{\alpha}{\beta}\log \left(\frac{\beta S - \mu}{\beta S_0 - \mu} \right)\right)  \, ,
\end{equation}
where we have fixed the level set $\mathcal C = 0$ onto which the dynamics takes place by means of the abovementioned initial conditions. From here we have that 
\begin{equation}
R=-\frac{\alpha}{\beta}\log \left(\frac{\beta S - \mu}{\beta S_0 - \mu} \right) .
\end{equation}
The conservation of the Hamiltonian (the total population is constant under the dynamics) allows us to write
\begin{equation}
I=1-S-R = 1-S + \frac{\alpha}{\beta}\log \left(\frac{\beta S - \mu}{\beta S_0 - \mu} \right) \, .
\end{equation}
Therefore, the first equation from \eqref{eq:SIRSendemic} reads
\begin{equation}
\dot S = \left( S - \frac{\mu}{\beta} \right) \left(- \beta + \beta S - \alpha\log \left(\frac{\beta S - \mu}{\beta S_0 - \mu} \right) \right) ,
\end{equation}
and therefore
\begin{equation}
\int^{S}_{S_0} \frac{d \zeta}{\left(\zeta - \frac{\mu}{\beta} \right) \left(- \beta + \beta \zeta - \alpha\log \left(\frac{\beta \zeta - \mu}{\beta S_0 - \mu} \right) \right)} = t .
\end{equation}
Setting 
\begin{equation}
\Sigma (t) = \left( \int^{S}_{S_0} \frac{d \zeta}{\left(\zeta - \frac{\mu}{\beta} \right) \left(- \beta + \beta \zeta - \alpha\log \left(\frac{\beta \zeta - \mu}{\beta S_0 - \mu} \right) \right)} \right)^{-1} (t)
\end{equation}
we have that the solution of the SIRS endemic system \eqref{eq:SIRSendemic} is given by 
\begin{equation}
\begin{cases}
&S(t) = \Sigma (t) \\
&I(t) = 1-\Sigma (t) + \frac{\alpha}{\beta}\log \left(\frac{\beta \Sigma (t) - \mu}{\beta S_0 - \mu} \right) \\
&R(t) = -\frac{\alpha}{\beta}\log \left(\frac{\beta \Sigma (t) - \mu}{\beta S_0 - \mu} \right)
\end{cases} .
\end{equation}
\hfill$\diamondsuit$
\end{example}

\begin{example}[Analytical solution of the SIR model with vaccination rate $\propto I$]

The procedure to find an exact solution to the system from Example \ref{ex:SIRvaccination} a) is similar. Recall from Proposition \ref{pr:SIRvaccI_biham} that the Casimir reads 
\begin{equation}
\mathcal C = -\left( R+\frac{(\alpha+v)}{\beta}\log \left(\frac{\beta S + v}{\beta S_0 + v} \right) \right) ,
\end{equation}
so 
\begin{equation}
R= - \frac{(\alpha+v)}{\beta}\log \left(\frac{\beta S + v}{\beta S_0 + v} \right) .
\end{equation}
The very same integration procedure shows that the solution of the system from Example \ref{ex:SIRvaccination} a) reads
\begin{equation}
\begin{cases}
&S(t) = \Sigma (t) \\
&I(t) = 1-\Sigma (t) + \frac{(\alpha+v)}{\beta}\log \left(\frac{\beta \Sigma (t) + v}{\beta S_0 + v} \right)  \\
&R(t) = - \frac{(\alpha+v)}{\beta}\log \left(\frac{\beta \Sigma (t) + v}{\beta S_0 + v} \right)
\end{cases} ,
\end{equation}
where now $\Sigma (t)$ is defined by
\begin{equation}
\Sigma (t) = \left( \int^{S}_{S_0} \frac{d \zeta}{(\zeta + \frac{v}{\beta}) \left( -\beta+ \beta \zeta- (\alpha+v) \log \frac{\beta \zeta + v}{\beta S_0 + v} \right) } \right)^{-1} (t) .
\end{equation}
\hfill$\diamondsuit$
\end{example}

\begin{example}[Analytical solution of the SIR model with vaccination rate $\propto S$]

The system from Example \ref{ex:SIRvaccination} b) presents a further complication, and we can only obtain the analytical solution in terms of two inverse functions. From Proposition \ref{pr:SIRvaccS_biham} we have that the Casimir, with the appropriate initial conditions, reads
\begin{equation}
\mathcal C = R+\frac{\alpha}{\beta} \log \frac{S}{S_0} - \frac{v}{\beta} \log \frac{1-S-R}{I_0}. 
\end{equation}
If we denote the solution of the equation $\mathcal C = 0$ by $R(t)=\rho(S(t))$, we have that
\begin{equation}
\Sigma(t) = \left( \int_{S_0}^S \frac{d \zeta}{- \beta \zeta (1-\zeta- \rho(\zeta)) - v \zeta} \right)^{-1} (t),
\end{equation}
and the solution of the system from Example \ref{ex:SIRvaccination} b) is
\begin{equation}
\begin{cases}
&S(t) = \Sigma (t) \\
&I(t) = 1-\Sigma (t) - \rho(t)  \\
&R(t) = \rho(\Sigma (t))
\end{cases} .
\end{equation}
\hfill$\diamondsuit$
\end{example}

Note that in all these examples the limit $\mu \to 0$ or $v \to 0$ is well defined and leads to the analytic solution to the original SIR system given in \cite{Miller2012,Harko2014}. This is a general fact for the generalized SIR system \eqref{eq:generalizedSIR}, since it can be thought of as an integrable deformation of the original SIR model \eqref{eq:originalSIR}, and therefore it can be expressed as a Hamiltonian system for a Poisson structure which is a deformation of  \eqref{eq:Poisson_original_SIR}. In turn, the corresponding Casimir functions are always deformations of \eqref{eq:H_2_originalSIR} in terms of the parameters $\mu$ and $v$.

\section{Concluding remarks}

Compartmental models are widely used to model the dynamics of epidemics. In this work we have proved that any compartmental model with constant population is a Hamiltonian system in which the total population plays the role of the Hamiltonian function. This  general result implies that all the analytical and geometrical tools connected with Hamiltonian systems can be applied in order to tackle epidemiological dynamics. Moreover, we have also showed that a large number of models with non-constant population can be transformed into effective models with constant population, so all the previous results hold also for them.

We have also found that compartmental models can be coupled, by keeping a Hamiltonian structure, in order to define interacting systems among different populations whose individuals can be interchanged. Moreover, three models which are generalizations of the original Kermack and McKendrick model have been proven to be bi-Hamiltonian, and their interacting analogues are shown to give rise to large families of bi-Hamiltonian systems.

As a direct application of the Hamiltonian structures here presented, we have obtained the exact analytic solutions for these three bihamiltonian models by making use of Casimir function of each corresponding Poisson algebra. This method indeed leads to the solutions found in \cite{Harko2014} for certain limiting cases, which are obtained provided the appropriate parameters vanish. Such Poisson-algebraic approach to analytical solutions for epidemiological models is indeed worth to be explored for other systems here presented, and becomes specially useful when the dimension and the number of parameters of the model increase, since analytic solutions have the potential to greatly simplify the analysis of the dynamical role played by each parameter. We stress that this approach  strongly relies upon the characterization and explicit computation of the Casimir functions of the Poisson structure associated to each model, a problem which is currently under investigation.

Moreover, the Hamiltonian nature of compartmental systems opens the path to the use of symplectic integrators \cite{Yoshida1990,Sanz-Serna1992} as a natural and  efficient tool in order to face the integration problem for these systems. Therefore, the use of the latter as benchmarks in order to compare the accurateness and efficiency of ordinary integrators versus symplectic integrators is also worth to be investigated.

Finally, the Hamiltonian structure here presented allows us to apply the well-known machinery of integrable deformations of Hamiltonian systems in order to define new compartmental models as integrable deformations of the dynamical systems here studied, including coupled ones. In particular, deformations based on Poisson coalgebras \cite{BR1998systematic} provide a suitable arena in order to face this problem by following the same techniques that have been previously used, for instance, in order to get integrable deformations and coupled versions of Lotka-Volterra, Lorenz, R\"ossler and Euler top dynamical systems \cite{BBM2011lotkavolterra,BBM2016rosslerlorentz,BMR2017bihamiltonian}. Work on this line is in progress and will be presented elsewhere.


\section*{Appendix. Hamiltonian and bi-Hamiltonian systems}

Let $M$ be a smooth manifold and let us denote its algebra of smooth functions by $\mathcal{C}^\infty (M)$. A Poisson structure on $M$ is given by a bivector field $\pi \in \Gamma (\bigwedge^{2} TM)$ such that $[\pi,\pi]=0$, where $[\cdot,\cdot] : \Gamma (\bigwedge^{2} TM) \times \Gamma (\bigwedge^{2} TM) \to \Gamma (\bigwedge^{3} TM)$ is the Schouten-Nijenjuis bracket. A \emph{Poisson manifold} $(M,\pi)$ is a manifold endowed with a Poisson structure. The Poisson bivector defines a Lie algebra structure on $\mathcal{C}^\infty (M)$ by means of the Poisson bracket 
\begin{equation}
\{ \cdot,\cdot \} : \mathcal{C}^\infty (M) \times \mathcal{C}^\infty (M) \to \mathcal{C}^\infty (M) ,
\end{equation}
defined by $\{ f_1,f_2 \} = (\mathrm{d}f_1 \otimes \mathrm{d}f_2) (\pi)$ for any $f_1,f_2 \in \mathcal{C}^\infty (M)$. 

A Poisson structure on $M$ induces a vector bundle morphism $\tilde \pi^\sharp : T^* M \to T M$ from the cotangent bundle to the tangent bundle of $M$. The linear morphism $\tilde \pi^\sharp_m : T^*_m M \to T_m M$ is not in general an isomorphism, since the Poisson structure can be degenerate. We call the rank of the Poisson structure $\pi$ at $m \in M$ to the rank of the linear map $\tilde \pi^\sharp_m$, which is always even due to the skew-symmetry of $\pi$. This vector bundle morphism defines a morphism between the respective spaces of sections of such bundles $\Gamma (T^* M) = \Omega^1 (M)$ and $\Gamma(T M)=\mathfrak X(M)$, which we denote by $\pi^\sharp : \Omega^1 (M) \to \mathfrak X(M)$. This morphism induces a map from functions to vector fields, given by
\begin{equation}
\begin{split}
\pi^\sharp \circ \mathrm{d} : \mathcal{C}^\infty (M) &\to \mathfrak X(M) \\
f & \to \pi^\sharp( \mathrm{d} f) = X_f
\end{split}
\end{equation} 
for any $f \in \mathcal{C}^\infty (M)$. The vector field $X_f$ is called the \emph{Hamiltonian vector field} associated to $f$. We say that $C \in \mathcal{C}^\infty (M)$ is a \emph{Casimir function} if its exterior derivative belongs to the kernel of $\pi^\sharp$, i.e. $\pi^\sharp(dC)=0$. In terms of the Poisson bracket we have that $\{ C,f \} = 0$ for all $f \in \mathcal{C}^\infty (M)$. 

Given a Poisson manifold $(M,\pi)$ and a smooth function $\mathcal H : M \to \mathbb R$ called the Hamiltonian, the dynamical evolution of any any function is given by the \emph{Hamilton's equations}
\begin{equation}
\dot f = \langle \pi^\sharp (\mathrm{d} \mathcal H), f \rangle.
\end{equation}
A Poisson manifold together with a Hamiltonian function is called a \emph{Hamiltonian system}. In this way, it is clear that the evolution of a Hamiltonian system is given geometrically by the integral curves of the Hamiltonian vector field $X_f$. In terms of the Poisson bracket, Hamilton's equations read
\begin{equation}
\dot f = \{f, \mathcal H\} .
\end{equation}

Given two Poisson structures $\pi_1$ and $\pi_2$ on the same manifold $M$ we say that they are \emph{compatible} if $\pi = \pi_1+ \pi_2$ is a Poisson structure. In fact, if there are two compatible Poisson structures on $M$ we can find an infinite number of them, since 
\begin{equation}
\pi = (1-\lambda) \pi_1+ \lambda \pi_2 \,
\end{equation}
is a Poisson structure. This family of Poisson structures is called the \emph{Poisson pencil}.

\begin{definition}
Let $\pi_1$ and $\pi_2$ be two compatible Poisson structures on $M$. A Hamiltonian system is called a \emph{bi-Hamiltonian system} if there exist two Hamiltonian functions $\mathcal H_1$ and $\mathcal H_2$ such that their Hamiltonian vector fields coincide. 
\end{definition}

Essentially, bi-Hamiltonian systems are Hamiltonian systems whose (unique) dynamics can be endowed with two different Hamiltonian structures. We have that
\begin{equation}
X_{\mathcal H_1} = \pi^\sharp_1 (\mathrm{d} \mathcal H_1) = \pi^\sharp_2 (\mathrm{d} \mathcal H_2) = X_{\mathcal H_2}
\end{equation}
Therefore, Hamilton's equations read
\begin{equation}
\dot f = \langle \pi^\sharp_1 (\mathrm{d} \mathcal H_1), f \rangle = \langle \pi^\sharp_2 (\mathrm{d} \mathcal H_2) , f \rangle ,
\end{equation}
or equivalently, 
\begin{equation}
\dot f = \{f, \mathcal H_1\}_1 = \{f, \mathcal H_2\}_2.
\end{equation}

Among Hamiltonian systems, those that posses a sufficient number of conserved quantities are particularly interesting, since these quantities can be used to find the solutions of the system. 

\begin{definition}
Let $(M,\pi)$ be a Poisson manifold such that $\mathrm{rank} \, \pi = 2n \leq \mathrm{dim}\,M$ on a dense open subset $\mathcal D$ of $M$. A Hamiltonian system on $M$ is \emph{completely integrable} (in the Liouville sense) if there exist $n-1$ functions $f_1,\ldots,f_{n-1} \in \mathcal{C}^\infty (M)$ such that the set $\{\mathcal H,f_1,\ldots,f_{n-1} \}$ is functionally independent, i.e.
\begin{equation}
\mathrm{d} \mathcal H \wedge \mathrm{d} f_1 \wedge \ldots \wedge \mathrm{d} f_{n-1} \neq 0
\end{equation}
on $\mathcal D$, and Poisson commuting, i.e.
\begin{equation}
\{ \mathcal H, f_a \} = 0, \qquad \{ f_a, f_b \} = 0 
\end{equation}
for all $a,b \in \{1, \ldots, n-1 \}$. 

Furthermore, the Hamiltonian system is called \emph{maximally superintegrable} if, additionally, there exist $g_1,\ldots,g_{n-1} \in \mathcal{C}^\infty (M)$ such that
\begin{equation}
\mathrm{d} \mathcal H \wedge \mathrm{d} f_1 \wedge \ldots \wedge \mathrm{d} f_{n-1} \wedge \mathrm{d} g_1 \wedge \ldots \wedge \mathrm{d} g_{n-1} \neq 0 ,
\end{equation}
and 
\begin{equation}
\{ \mathcal H, g_a \} = 0 .
\end{equation}
\end{definition}

Functions that Poisson commute with the Hamiltonian are called additional first integrals. Therefore, a system is completely integrable if it admits $n-1$ functionally independent additional first integrals, while it is  maximally superintegrable if it admits $2n-2$. For completely integrable systems, Liouville-Arnold theorem guarantees that the solution can be found by quadratures if the canonical transformation to action-angle coordinates is known. Maximally superintegrable systems are particular cases of completely integrable systems for which the dynamics is strictly unidimensional.

For Hamiltonian systems defined on an arbitrary Poisson manifold $(M,\pi)$ the notions of complete and maximal superintegrability are intimately related to the Casimir structure of the Poisson manifold, which is a geometrical property of $(M,\pi)$ independent of the particular Hamiltonian system under consideration. In particular, any Casimir function is a first integral for any Hamiltonian system defined on $(M,\pi)$.

\section*{Acknowledgements}

This work has been partially supported by Agencia Estatal de Investigaci\'on (Spain) under grants MTM2016-79639-P (AEI/FEDER, UE) and PID2019-106802GB-I00 / AEI / 10.13039/501100011033, and by Junta de Castilla y Le\'on (Spain) under grants BU229P18 and BU091G19.


\end{document}